\documentclass[11pt]{amsart}
\usepackage{url, calc}
\usepackage{amsmath,amsxtra,amssymb,latexsym,epsfig,amscd,amsthm,fancybox,epsfig}
\usepackage[mathscr]{eucal}
\usepackage{graphicx}
\usepackage{multicol,xcolor}
\usepackage{epsfig} 
\usepackage{epstopdf}
\usepackage{cases}
\usepackage{subfig}
\usepackage{color}
\usepackage{hyperref}
\usepackage{bigints}
\newtheorem{thm}{Theorem}[section]
\newtheorem {asp}{Assumption}[section]
\newtheorem{lm}{Lemma}[section]
\newtheorem{rmk}{Remark}[section]

\newtheorem{deff}{Definition}[section]

\newtheorem{prop}{Proposition}[section]
\theoremstyle{definition}

\theoremstyle{remark}

\newtheorem{example}{Example}[section]
\numberwithin{equation}{section}

\DeclareMathOperator{\suppo}{supp}
\DeclareMathOperator{\Conv}{Conv}
\newcommand{\eps}{\varepsilon}

\newcommand{\M}{\mathcal{M}}
\newcommand{\F}{\mathcal{F}}

\newcommand{\E}{\mathbb{E}}
\newcommand{\BE}{\mathbf{E}}
\newcommand{\BB}{\mathbf{B}}

\newcommand{\BX}{\mathbf{X}}
\newcommand{\bx}{\mathbf{x}}

\newcommand{\by}{\mathbf{y}}

\newcommand{\bc}{\mathbf{c}}

\newcommand{\bp}{\mathbf{p}}
\newcommand{\bq}{\mathbf{q}}

\newcommand{\N}{\mathbb{N}}

\newcommand{\PP}{\mathbb{P}}

\newcommand{\K}{\mathcal{K}}

\newcommand{\CH}{{\mathcal{H}_\mu}}
\newcommand{\R}{\mathbb{R}}

\newcommand{\Lom}{\mathcal{L}}
\newcommand{\U}{\mathcal{U}}

\newcommand{\wtd}{\widetilde}
\numberwithin{equation}{section}

\newcommand{\I}{\mathcal{I}}

\newcommand{\bed}{\begin{displaymath}}
\newcommand{\eed}{\end{displaymath}}
\newcommand{\bea}{\bed\begin{array}{rl}}
\newcommand{\eea}{\end{array}\eed}

\newcommand{\barray}{\begin{array}{ll}}
\newcommand{\earray}{\end{array}}
\newcommand{\diag}{{\rm diag}}

\newcommand{\1}{\boldsymbol{1}}
\newcommand{\0}{\boldsymbol{0}}
\newcommand{\bdelta}{\boldsymbol{\delta}}

\newcommand{\dist}{\mathrm{dist}}

\def\hat{\widehat}
\def\a.s{\text{\;a.s.\;}}

\title[Stochastic Kolmogorov systems]{Coexistence and extinction for stochastic Kolmogorov systems}
\author[A. Hening]{Alexandru Hening }
\address{Department of Mathematics \\
 Imperial College London\\
 South Kensington Campus\\
  London, SW7 2AZ\\
 United Kingdom}
 \email{a.hening@imperial.ac.uk}

\author[D. Nguyen]{Dang H. Nguyen }
\thanks{D. Nguyen was in part supported by
 the National Science Foundation
under grant DMS-1207667.}
\address{Department of Mathematics \\
 Wayne State University\\
 Detroit, MI 48202 \\
 United States}
 \email{dangnh.maths@gmail.com}


\begin{document}
\begin{abstract}
In recent years there has been a growing interest in the study of the dynamics of stochastic populations. A key question in population biology is to understand the conditions under which populations coexist or go extinct. Theoretical and empirical studies have shown that coexistence can be facilitated or negated by both biotic interactions and environmental fluctuations. We study the dynamics of $n$ populations that live in a stochastic environment and which can interact nonlinearly (through competition for resources, predator-prey behavior, etc.). Our models are described by $n$-dimensional Kolmogorov systems with white noise (stochastic differential equations - SDE). We give sharp conditions under which the populations converge exponentially fast to their unique stationary distribution as well as conditions under which some populations go extinct exponentially fast.

The analysis is done by a careful study of the properties of the invariant measures of the process that are supported on the boundary of the domain. To our knowledge this is one of the first general results describing the asymptotic behavior of stochastic Kolmogorov systems in non-compact domains.

We are able to fully describe the properties of many of the SDE that appear in the literature.
In particular, we extend results on two dimensional Lotka-Volterra models, two dimensional predator-prey models, $n$ dimensional simple food chains, and two predator and one prey models. We also show how one can use our methods to classify the dynamics of any two-dimensional stochastic Kolmogorov system satisfying some mild assumptions.
\end{abstract}
\maketitle
\tableofcontents
\section{Introduction}

Real populations do not evolve in isolation and as a result much of ecology is concerned with
understanding the characteristics that allow two species to coexist, or one species to take over the
habitat of another. It is of fundamental importance to understand what will happen to an invading
species. Will it invade successfully or die out in the attempt? If it does invade, will it coexist
with the native population? Mathematical models for invasibility have contributed significantly to
the understanding of the epidemiology of infectious disease outbreaks (\cite{cross05}) and ecological
processes (\cite{law96}; \cite{cas01}). There is widespread empirical evidence
that heterogeneity, arising from abiotic (precipitation, temperature, sunlight) or biotic (competition,
predation) factors, is important in determining invasibility (\cite{davies05}; \cite{pyvsek05}). The fluctuations of the environment make the dynamics of populations inherently stochastic.

The combined effects of biotic interactions and environmental fluctuations are key when trying to determine species richness. Sometimes biotic effects can result in species going extinct. However, if one adds the effects of a random environment extinction might be reversed into coexistence. In other instances deterministic systems that coexist become extinct once one takes into account environmental fluctuations. A  successful way of studying this interplay is by modelling the populations as discrete or continuous time Markov processes and looking at the long-term behavior of these processes (\cite{C00, ERSS13, EHS15,  LES03, SLS09, SBA11, BEM07, BS09, BHS08, CM10, CCA09}).

A natural way of analyzing the coexistence of species is by analyzing the average per-capita growth rate of a population when rare. Intuitively, if this growth rate is positive the respective population increases when rare, and can invade, while if it is negative the population decreases and goes extinct. If there are only two populations then coexistence is ensured if each population can invade when it is rare and the other populaton is stationary (\cite{T77, CE89, EHS15}).

There is a general theory for coexistence for deterministic models (\cite{H81, H84, HJ89}). It is shown that a sufficient condition for persistence is the existence of a fixed set of weights associated with the interacting populations such that this weighted combination of the populations's invasion rates is positive for any invariant measure supported by the boundary (i.e. associated to a sub-collection of populations) - see \cite{H81}.

A few recent studies have explored the effect of environmental stochasticity on continuous-time models. In \cite{BHS08} the authors found that if a deterministic continuous-time model satisfies the above persistence criterion then under some weak assumptions the corresponding stochastic differential equation with a small diffusion term has a positive stationary distribution concentrated on the positive global attractor of the deterministic system. For general stochastic difference and differential equations with arbitrary levels of noise on a \textit{compact} state space sufficient conditions for persistence are given in \cite{SBA11}.

The aim of this paper is two-fold.
First, we want to have a general theory that gives sharp sufficient conditions for both persistence \textit{and} extinction for stochastic Kolmogorov systems.
Second, we want our methods to work on non-compact state spaces (for example $\R_+^n$).

The criteria we present for persistence are the same as those in \cite{SBA11}. However, we extend their result to non-compact state spaces and we prove that the convergence rate is exponential. We note that some of our persistence results have been announced in the $2014$ Bernoulli lecture of Michel Bena\"{\i}m. Furthermore, criteria for persistence for general Markov processes appear in \cite{B14} and we use some of those ideas in our proofs.
We come up with natural assumptions under which one or more populations go extinct with nonzero probability. There do not seem to be general criteria for extinction in the literature. Results have been obtained for a Lotka-Volterra competitive system in the two-dimensional setting for SDE (\cite{EHS15}) and piecewise-deterministic Markov processes (\cite{BL16}). However, in these cases there are only two or three ergodic invariant probability measures on the boundary and as such the proofs simplify significantly.

It should be noted that most of the related results in the literature
are obtained by choosing a function and imposing conditions
such that the function has some Lyapunov-type properties.
The choice of Lyapunov function is usually artificial and imposes unnecessary constraints on the system.
The results one gets are therefore limited as the particular Lyapunov function does not reflect the true nature of the dynamical system.
Our approach is to carefully analyze the dynamics of the process
near the boundary of its domain. Because of this, we are able to fully characterize and classify the asymptotic behavior of the system.

As corollaries of our main theorems,
we extend results on two dimensional Lotka-Volterra models (see \cite{EHS15, BL16}), two dimensional predator-prey models (see \cite{RP07, R03, CK05}), two predator and one prey models (see \cite{LB16}) and populations modeled by SDE in a compact state space (see \cite{SBA11}).

The paper is organized as follows. In section \ref{s:results} we define our framework, the problems we study, our different assumptions and the main results.  In Section \ref{s:examp} we exhibit a few examples that fall into our general setting (Lotka-Volterra competition and predator-prey models). We also give an example of a cooperative Lotka-Volterra model that does not satisfy our assumptions. However, in this case either the solution blows up in finite time or there is no invariant probability measure supported by the interior of the domain.
In Section \ref{s:prep} we analyze some of the properties of the SDE that models our populations. In particular we show it has a well-defined strong solution $\BX$ for all $t>0$ and that this solution is pathwise unique. Section \ref{s:perm} is devoted to the study of conditions under which $\BX$ converges to its unique invariant probability measure on $\R_+^{n,\circ}:=(0,\infty)^n$. In Theorem \ref{thm3.1} we show that, under some natural assumptions, $\BX$ is strongly stochastically persistent and that the convergence in total variation of its transition probability to a unique stationary distribution on $\R_+^{n,\circ}$ is exponentially fast. In Section \ref{s:extin} we look at when one or more of the populations go extinct with a positive probability. First, we show in Theorem \ref{thm4.1} that if there exists an invariant probability measure living on the boundary that is a \textit{sink}, then the process converges to the boundary in a weak sense. Under a few extra assumptions we show in Theorem \ref{thm4.2} that for every sink invariant measure $\mu$ on the boundary the process converges with strictly positive probability to the support of $\mu$.
Finally, we present in the Appendix the proofs of some auxiliary lemmas from Section \ref{s:prep} and Section \ref{s:extin}.
\subsection{Notation and Results} \label{s:results}

We work on a complete probability space $(\Omega,\F,\{\F_t\}_{t\geq0},\PP)$ with a filtration $\{\F_t\}_{t\geq 0}$ satisfying the usual conditions.
Consider a stochastic Kolmogorov system

\begin{equation}\label{e:system}
dX_i(t)=X_i(t) f_i(\BX(t))\,dt+X_i(t)g_i(\BX(t))\,dE_i(t), ~i=1,\dots,n
\end{equation}
taking values in $[0,\infty)^n$. We assume
$\BE(t)=(E_1(t),\dots, E_n(t))^T=\Gamma^\top\BB(t)$ where
$\Gamma$ is a $n\times n$ matrix such that
$\Gamma^\top\Gamma=\Sigma=(\sigma_{ij})_{n\times n}$
and $\BB(t)=(B_1(t),\dots, B_n(t))$ is a vector of independent standard Brownian motions adapted to the filtration $(\F_t)_{t\geq 0}$. The SDE \eqref{e:system} is describing the dynamics of $n$ interacting populations $\BX(t)=(X_1(t),\dots,X_n(t))_{t\geq 0}$. Throughout the paper we set $\R_+^n:=[0,\infty)^n$ and $\R_+^{n,\circ}:=(0,\infty)^n$.

\begin{rmk}
One might wonder if one could treat the more general model

$$dX_i(t)=X_i(t)f_i(\mathbf X(t))\,dt+X_i(t)g_i(\mathbf X(t))\sum_{j=1}^n \Sigma_{ij}(\mathbf X(t))\,dB_j(t), ~i=1,\dots,n$$

In our model \eqref{e:system}, we work with a constant correlation matrix $\Sigma=(\sigma_{ij})$
but it can be seen that the proofs
do not depend on whether $\Sigma$ is constant or a function of $\mathbf x$.
Thus, our results still hold if $\Sigma$ depends on $\mathbf x$ as long as it is bounded and locally Lipschitz.
Actually we can always assume it is bounded  because we can normalize $\Sigma$
and absorb the necessary factors into $g_i(\mathbf x)$.

\end{rmk}

The drift term of our system is due to the deterministic dynamics while the diffusion term is due to the effects of random fluctuations of the environment. The drift for population $i$ is given by $X_i(t) f_i(\BX(t))$ where $f_i$ is its per-capita growth rate.  From now on the process given by the solution to \eqref{e:system} will be denoted by $\BX$ or $(\BX(t))_{t\geq 0}$.

Let $\Lom$ be the infinitesimal generator of the process $\BX$. For smooth enough functions $F:\R_+^n\to \R$ the generator $\Lom$ acts as
\[
\Lom F(\bx) = \sum_i x_if_i(\bx)\frac{\partial F}{\partial x_i}(\bx) + \frac{1}{2}\sum_{i,j}\sigma_{ij}x_ix_jg_i(\bx)g_j(\bx)\frac{\partial^2 F}{\partial x_i \partial x_j}(\bx).
\]

We use the norm $\|\bx\|=\sum_{i=1}^n |x_i|$ in $\R^n$. For $a,b\in\R$, let $a\wedge b:=\min\{a,b\}$ and $a\vee b:=\max\{a,b\}$. Similarly we let $\bigwedge_{i=1}^n u_i:=\min_{i}u_i$ and  $\bigvee_{i=1}^n u_i:=\max_{i}u_i$.

We remark that \eqref{e:system} can be seen as a generalization to non-compact state spaces of the model studied in \cite{SBA11}. The following is a standing assumption throughout the paper.
\begin{asp}\label{a.nonde}  The coefficients of \eqref{e:system} satisfy the following conditions:
\begin{itemize}
\item[(1)] $\diag(g_1(\bx),\dots,g_n(\bx))\Gamma^\top\Gamma\diag(g_1(\bx),\dots,g_n(\bx))=(g_i(\bx)g_j(\bx)\sigma_{ij})_{n\times n}$
is a positive definite matrix for any $\bx\in\R^{n}_+$.
\item[(2)] $f_i(\cdot), g_i(\cdot):\R^n_+\to\R$ are locally Lipschitz functions for any $i=1,\dots,n.$
\item[(3)] There exist $\bc=(c_1,\dots,c_n)\in\R^{n,\circ}_+$ and $\gamma_b>0$ such that
\begin{equation}\label{a.tight}
\limsup\limits_{\|x\|\to\infty}\left[\dfrac{\sum_i c_ix_if_i(\bx)}{1+\bc^\top\bx}-\dfrac12\dfrac{\sum_{i,j} \sigma_{ij}c_ic_jx_ix_jg_i(\bx)g_j(\bx)}{(1+\bc^\top\bx)^2}+\gamma_b\left(1+\sum_{i} (|f_i(\bx)|+g_i^2(\bx))\right)\right]<0.
\end{equation}
\end{itemize}
\end{asp}
\begin{rmk}
Parts (2) and (3) of Assumption \ref{a.nonde} guarantee the existence and uniqueness of strong solutions to \eqref{e:system}.
We need part (1) of Assumption \ref{a.nonde} to ensure that the solution to \eqref{e:system} is a non degenerate diffusion.
Moreover, we show later that (3) implies the tightness of the family of transition probabilities
of the solution to \eqref{e:system}.
\end{rmk}

\begin{rmk}
There are a few different ways to add stochastic noise to deterministic population dynamics. We assume that the environment mainly affects the growth/death rates of the populations. See \cite{T77, B02, G88, ANY, EHS15, ERSS13, SBA11} for more details.
\end{rmk}

We next define what we mean by persistence and extinction in our setting.
\begin{deff}
The process $\BX$ is strongly stochastically persistent if it has a unique invariant probability measure $\pi^*$ on $\R^{n,\circ}_+$ and
\begin{equation}
\lim\limits_{t\to\infty} \|P_\BX(t, \mathbf{x}, \cdot)-\pi^*(\cdot)\|_{\text{TV}}=0, \;\mathbf{x}\in\R^{n,\circ}_+
\end{equation}
where $\|\cdot,\cdot\|_{\text{TV}}$ is the total variation norm and $P_\BX(t,\mathbf{x},\cdot)$ is the transition probability of $(\BX(t))_{t\geq 0}$.
\end{deff}

\begin{deff}
If $\BX(0)=\bx\in \R^{n,\circ}_+ $ we say the population $X_i$ goes extinct with probability $p_\bx>0$ if
\[
\PP_\bx\left\{\lim_{t\to\infty}X_i(t)=0\right\}=p_\bx.
\]
We say the population $X_i$ goes extinct if for all $\bx\in\R^{n,\circ}_+$
\[
\PP_\bx\left\{\lim_{t\to\infty}X_i(t)=0\right\}=1.
\]
\end{deff}

\begin{example}
Most of the common ecological models satisfy condition \eqref{a.tight}.
\begin{itemize}

\item Consider the linear one-dimensional model
$$dX(t)=aX(t)dt+\sigma X(t)dB(t).$$
If $a-\frac{\sigma^2}2<0$ then
\eqref{a.tight} is satisfied for any $c>0$.

\item Consider the logistic model
$$dX(t)=X(t)[a-bX(t)]dt+\sigma X(t)dB(t), b>0.$$ Then equation
\eqref{a.tight} is satisfied for any $c>0$.

\item Consider the competitive Lotka-Volterra model
$$dX_i(t)=X_i(t)\left[a_i-\sum_{j}b_{ji}X_j(t)\right]dt+X_i(t)g_i(\BX(t))dE_i(t),$$
with $b_{ji}> 0, j,i=1,\dots,n$.

If $\sum_{i=1}^n g_i^2(\bx)<K(1+\|\bx\|+\bigwedge_{i=1}^n g_i^2(\bx))$ then
 \eqref{a.tight} is satisfied with $\bc=(1,\dots,1)$. We give a short argument for why this is true.
 Since $b_{ij}>0$, there is $\tilde b>0$ such that
  \begin{equation}\label{e1-cLV}\dfrac{\sum_i x_i(a_i-\sum_j b_{ji}x_j)}{1+\sum_i x_i}<-\tilde b (1+\sum_i x_i)
   \end{equation}
 if $\|\bx\|$ is sufficiently large. By the Cauchy-Schwarz inequality, there are $\tilde\sigma_1,\tilde\sigma_2>0$ such that
 \begin{equation}\label{e2-cLV}
 -\dfrac12\dfrac{\sum_{i,j} \sigma_{ij}x_ix_jg_i(\bx)g_j(\bx)}{(1+\sum_i x_i)^2}\leq -\tilde\sigma_1\dfrac{\sum_{i} x_i^2g_i^2(\bx)}{(1+\sum_i x_i)^2}\leq  -\tilde\sigma_2\left(\bigwedge_{i=1}^n g^2_{i}(\bx)\right),
 \end{equation}
 when $\|\bx\|$ is sufficiently large.
 In light of \eqref{e1-cLV} and \eqref{e2-cLV}, if $\sum_{i=1}^n g_i^2(\bx)<K(1+\|\bx\|+\bigwedge_{i=1}^n g_i^2(\bx))$, we can find $\gamma_b>0$ such that
\begin{align*}
\dfrac{\sum_i x_i(a_i-\sum_j b_{ji}x_j)(\bx)}{1+\sum_i x_i}&-\dfrac12\dfrac{\sum_{i,j} \sigma_{ij}x_ix_jg_i(\bx)g_j(\bx)}{(1+\sum_i x_i)^2}\\
&+\gamma_b\left(1+\sum_{i} (|a_i-\sum_j b_{ji}x_j|+g_i^2(\bx))\right)<0
\end{align*}
for sufficiently large $\|\bx\|$.
As a result \eqref{a.tight} holds.
\item Consider predator-prey Lotka-Volterra model
$$
\begin{cases}
dX(t)=X(t)[a_1-b_1X(t)-c_1Y(t)]dt+X(t)g_1(X(t), Y(t))dE_1(t)\\
dY(t)=Y(t)[-a_2-b_2Y(t)+c_2X(t)]dt+Y(t)g_2(X(t), Y(t))dE_2(t),\\
\end{cases}
$$
with $b_1,b_2>0, c_1,c_2\geq 0, a_2\geq 0$.
If $\sum_{i=1}^2 g_i^2(x,y)<K(1+x+y+g_1^2(x,y)\wedge g_2^2(x,y))$,
 one can use arguments similar to those from the competitive Lotka-Volterra model to show that
 \eqref{a.tight} is satisfied with $\bc=(c_2,c_1)$.
\end{itemize}
\end{example}

Let $\M$ be the set of ergodic invariant probability measures of $\BX$ supported on the boundary $\partial\R^n_+:=\R_+^n\setminus \R_+^{n,\circ}$. Note that if we let $\bdelta^*$ be the Dirac measure concentrated at $\0$ then $\bdelta^*\in\M$ so that $\M\neq\emptyset$. For a subset $\wtd\M\subset \M$, denote by $\Conv(\wtd\M)$ the convex hull of $\wtd\M$,
that is the set of probability measures $\pi$ of the form
$\pi(\cdot)=\sum_{\mu\in\wtd\M}p_\mu\mu(\cdot)$
with $p_\mu>0,\sum_{\mu\in\wtd\M}p_\mu=1$.

Consider $\mu\in\M$.
Assume $\mu\neq \bdelta^*$. Since the diffusion $\BX$ is non degenerate in each subspace,
there exist $0<n_1<\dots< n_k\leq n$
such that $\suppo(\mu)= \R^\mu_+$ where
$$\R_+^\mu:=\{(x_1,\dots,x_n)\in\R^n_+: x_i=0\text{ if } i\in I_\mu^c\}$$
for
$I_\mu:=\{n_1,\dots, n_k\}$ and
$I_\mu^c:=\{1,\dots,n\}\setminus\{n_1,\dots, n_k\}$. If $\mu=\bdelta^*$ then we note that $\R^{\bdelta^*}_+=\{\0\}$.
Let
$$\R_+^{\mu,\circ}:=\{(x_1,\dots,x_n)\in\R^n_+: x_i=0\text{ if } i\in I_\mu^c\text{ and }x_i>0\text{ if  }x_i\in I_\mu\}$$ and $\partial\R_+^{\mu}:=\R_+^\mu\setminus\R_+^{\mu,\circ}$.

The following condition ensures strong stochastic persistence.
\begin{asp}\label{a.coexn}
For any $\mu\in\Conv(\M)$ one has
$$\max_{\{i=1,\dots,n\}}\left\{\lambda_i(\mu)\right\}>0,$$
where
$$\lambda_i(\mu):=\int_{\partial\R^n_+}\left(f_i(\bx)-\dfrac{\sigma_{ii}g_i^2(\bx)}2\right)\mu(d\bx).$$
(In view of Lemma \ref{lm2.3},
$\lambda_i(\mu)$ is well-defined.)
\end{asp}
\begin{thm}\label{t:pers}
Suppose that Assumptions \ref{a.nonde} and \ref{a.coexn} hold. Then $\BX$ is strongly stochastically persistent and converges exponentially fast to its unique invariant probability measure $\pi^*$ on $\R_+^{n,\circ}$.
\end{thm}
The proof of this result is presented in detail in Section \ref{s:perm}.
The following remark gives a rough intuitive sketch of the proof.
\begin{rmk}
From a dynamical point of view,
the solution in the interior domain $\R^{n,\circ}_+$
is persistent if every invariant probability measure on the boundary
is a ``repeller''.
In a determistic setting,
an equilibrium is a repeller
if it has a positive Lyapunov exponent.
In a stochastic model,
ergodic invariant measures play a similar role.
To determine the Lyapunov exponents of an ergodic invariant measure,
one can look at the equation for $\ln X_i(t)$. An application of It\^o's Lemma yields that		
$$
\dfrac{\ln X_i(t)}t=\dfrac{\ln X_i(0)}t+\dfrac1t\int_0^t\left[f_i(\BX(s))-\dfrac{g_i^2(\BX(s))\sigma_{ii}}2\right]ds+\dfrac1t\int_0^t g_i(\BX(s))dE_i(s).
$$

If $\BX$ is close to the support of an ergodic invariant measure $\mu$ for a long time,
then
$$\dfrac1t\int_0^t\left[f_i(\BX(s))-\dfrac{g_i^2(\BX(s))\sigma_{ii}}2\right]ds$$
can be approximated by the average with respect to $\mu$
$$\lambda_i(\mu)=\int_{\partial \R^n_+}\left(f_i(\bx)-\dfrac{g_i^2(\bx)\sigma_{ii}}2\right)\mu(d\bx)$$
while the term $$\dfrac{\ln X_i(0)}t+\dfrac1t\int_0^t g_i(\BX(s))dE_i(s)$$ is negligible.
This implies that $\lambda(\mu_i), i=1,\dots, n$ are the Lyapunov exponents of $\mu$
(it can also be seen that $\lambda(\mu_i)$ gives the long-term growth rate of $X_i(t)$
if $\BX$ is close to the support of $\mu$).
As a result,
if $\max_{i=1}^n  \{\lambda(\mu_i)\}>0$,
then the invariant measure $\mu$ is a ``repeller''.
Therefore, Assumption \ref{a.coexn} guarantees the persistence of the population.
Moreover, by evaluating the exponential rate $\dfrac{\ln X_i(T)}T$
for sufficiently large $T$ (so that the ergodicity takes effect),
we can show that the solution goes away from the boundary exponentially fast,
and then obtain a geometric rate of convergence in total variation under Assumptions \ref{a.nonde} and \ref{a.coexn}.
This is achieved by constructing a suitable Lyapunov function
with the help of the Laplace transform and
the approximations that were mentioned above.
Note that since we work on a non-compact space, Assumption \ref{a.coexn} part (3) is needed to show
that the solution enters a compact subset of $\R^n_+$ exponentially fast.
\end{rmk}

The following condition will imply extinction.
\begin{asp}\label{a.extn}
There exists a $\mu\in\M$ such that
\begin{equation}\label{ae3.1}
\max_{i\in I_\mu^c}\{\lambda_i(\mu)\}<0.
\end{equation}
If $\R_+^\mu\ne\{\0\}$, suppose further that
for any $\nu\in\Conv(\M_\mu)$ , we have
\begin{equation}\label{ae3.2}
\max_{i\in I_\mu}\{\lambda_i(\nu)\}>0
\end{equation}
where $\M_\mu:=\{\nu'\in\M:\suppo(\nu')\subset\partial\R_+^\mu\}.$
\end{asp}
Define
\begin{equation}\label{e:M1}
\M^1:=\left\{\mu\in\M : \mu ~~\text{satisfies Assumption} ~~\ref{a.extn}\right\},
\end{equation}
and
\begin{equation}\label{e:M2}
\M^2:=\M\setminus\M^1.
\end{equation}
\begin{thm}\label{t:ex}
Under Assumptions \ref{a.nonde} and  \ref{a.extn} for any $\delta>0$ sufficiently small
and any $\bx\in\R^{n,\circ}_+$ we have
\begin{equation}\label{e.extinction}
\lim_{t\to\infty}\E_\bx \left(\bigwedge_{i=1}^n X_i(t)\right)^{\delta}=0.
\end{equation}
\end{thm}
The proof of this result is given in Section \ref{s:extin}.
\begin{rmk}
If an ergodic invariant measure $\mu$ with support on the boundary is an ``attractor'',
it will attract solutions starting nearby.
Intuitively, condition \eqref{ae3.1} forces $X_i(t), i\in I_\mu^c$ to get close to $0$
if the solution starts close to $\R^{\mu,\circ}_+$.
We need condition \eqref{ae3.2} to ensure that
$\mu$ is a ``sink'' in $\R^{\mu,\circ}_+$,
that is, if $\BX$ is close to $\R^{\mu,\circ}_+$,
it is not  pulled away to the boundary $\partial\R^{\mu,\circ}_+$ of $\R^{\mu,\circ}_+$ (see Remark \ref{r:degenerate}).

To prove Theorem \ref{t:ex},
using the idea above,
we construct a Lyapunov function $U$ vanishing on $\R^{\mu,\circ}_+$ such that
$\E_\bx U(\BX(T))\leq U(\bx)$ if $\bx\in\R^{n,\circ}_+$ sufficiently close to $\R^{\mu,\circ}_+$ and
$T$ is a sufficiently large time.
Then, we can construct a supermartigale to show that with a large probability
$X_i(t), i\in I_\mu^c$ cannot  go far from $0$
if the starting point of $\BX$ is sufficiently close to $\R^{\mu,\circ}_+$.
With some additional arguments from the theory of Markov processes,
we can show that $\BX$ has no invariant probability measure in $\R^{n,\circ}_+$
and approaches the boundary in some sense.

In the case when there is no persistence
one may want to know exactly which species go extinct and which survive.
We answer this question in Theorem \ref{thm4.2}.
Relying on the repulsion of invariant measures in $\M^2 = \M\setminus \M^1$
and properties of the randomized occupation measures,
we can deduce that the process $\BX$ must
enter the ``attracting'' region of some invariant measure in $\M^1$.
Finally, the attraction property of the measures from $\M^1$ helps us characterize
the survival and extinction of each species.
\end{rmk}
\begin{rmk}\label{r:degenerate}
If condition \eqref{ae3.2} does  not hold we could have the following bad situation. Assume there exists $\nu\in\M_\mu$ such that
\[
\max_{i\in I_\mu\setminus I_\nu}\left\{\lambda_i(\nu)\right\}=0.
\]
In this case $\nu$ is not always a ``repeller''. Solutions that start near $\R_+^{\mu,\circ}$ will tend to stay close to $R^\mu_+$ since $\lambda_i(\mu)<0, i \in I_\mu^c$. However, if $\nu$ is not a repeller the solutions may concentrate on $\R^\nu\subset \partial \R_+^\mu$. Now, if there exists $i^*\in I_\mu^c$ such that $\lambda_{i^*}(\nu)>0$ then solutions can be pushed away from $\R_+^\mu$ since $X_{i^*}(t)$ will tend to increase.

\end{rmk}

To characterize the extinction of specific populations, we need some additional conditions.
\begin{asp}\label{a.extn2}
Suppose that there is $\delta_1>0$ such that
$$
\lim\limits_{\|\bx\|\to\infty} \dfrac{\|\bx\|^{\delta_1}\sum_i g_i^2(\bx)}{1+\sum_i(|f_i(\bx)|+|g_i(\bx)|^2)}=0.
$$
Without loss of generality, suppose that $\delta_1\leq\delta_0$ (where $\delta_0$ is defined at the beginning of Section \ref{s:prep}).
\end{asp}
\begin{rmk}
Assumption \ref{a.extn2} forces the growth rates of $g_i^2(\cdot)$
to be slightly lower than those of $|f_i(\cdot)|$.
This is needed in order to suppress the diffusion part so that
we can obtain the tightness of  the random normalized occupation measures
$$\wtd \Pi_t(\cdot):=\dfrac1t\int_0^t\1_{\{\BX(s)\in\cdot\}}ds,\,t>0$$
as well as  the convergence of $\int_{\R^n_+}\left(f_i(\bx)-\dfrac{g_i^2(\bx)}2\right)\wtd \Pi_{t_k}(d\bx)$
to
$\lambda_i(\pi)$
given that $\wtd \Pi_{t_k}$ converges weakly to $\pi$ for some sequence $(t_k)_{k\in\N}$ with $\lim_{k\to\infty} t_k=\infty$.
Having these properties, we can analyze the asymptotic behavior of the sample paths of the solution.
\end{rmk}

To describe exactly which populations go extinct,
we need an additional assumption which ensures that apart from those in $\Conv(\M^1)$,
invariant probability measures are ``repellers''.
\begin{asp}\label{a.extn3}
Suppose that one of the following is true
\begin{itemize}
  \item $\M^2=\emptyset$

  \item For any $\nu\in\Conv(\M^2)$, $\max_{\{i=1,\dots,n\}}\left\{\lambda_i(\nu)\right\}>0.$
\end{itemize}
\end{asp}
For any initial condition $\BX(0)=\bx\in\R^n_+$,
denote the weak$^*$-limit set of the family $\left\{\wtd \Pi_t(\cdot), t\geq 1\right\}$ by $\U=\U(\omega)$.

\begin{thm}\label{t:ex2}
Suppose that Assumptions \ref{a.nonde}, \ref{a.extn2} and \ref{a.extn3} are satisfied and $\M^1\neq \emptyset$.
Then for any $\bx\in\R^{n,\circ}_+$
\begin{equation}\label{e0-thm4.2}
\sum_{\mu\in\M^1} P_\bx^\mu=1
\end{equation}
where
$$P_\bx^\mu:=\PP_\bx\left\{\U(\omega)=\{\mu\}\,\text{ and }\,\lim_{t\to\infty}\dfrac{\ln X_i(t)}t=\lambda_i(\mu)<0, i\in I_\mu^c\right\}>0, \bx\in\R^{n,\circ}_+, \mu\in \M^1.$$
\end{thm}

\begin{rmk}
Our results can be easily modified and applied to SDE living on smooth enough domains $D\subset \R^n$. We chose to work on $[0,\infty)^n$ because it was the most natural non-compact example for the dynamics of biological populations. In particular one can recover and extend the results from \cite{SBA11} where the authors looked at the state space $D=\{(y_1,\dots,y_n)\in\R^{n}_+: y_1+\dots+y_n=1\}$ .
\end{rmk}

\section{Examples}\label{s:examp}

We present some applications of our main results. We will make use of Theorems \ref{t:pers},\ref{t:ex}, and \ref{t:ex2} together with the following intuitive lemma whose proof is postponed to Section \ref{s:extin}.
\begin{lm}\label{l:int}
For any $\mu\in\M$ and $i\in I_\mu$ we have
$\lambda_i(\mu)=0.$
\end{lm}
\begin{rmk}
The intuition behind Lemma \ref{l:int} is the following: if we are inside the support of an ergodic invariant measure $\mu$ then we are at an `equilibrium' and the process does not tend to grow or decay.
\end{rmk}

\begin{example}\label{ex1}
Consider a stochastic Lotka-Voltera competitive model
\begin{equation}\label{e1-ex1}
\begin{cases}
dX_1(t)=X_1(t)[a_1-b_1X_1(t)-c_1X_2(t)]dt+X_1(t)dE_1(t)\\
dX_2(t)=X_2(t)[a_2-b_2X_2(t)-c_2X_1(t)]dt+X_2(t)dE_1(t)
\end{cases}
\end{equation}
where $b_i,c_i > 0, i=1,2$.
It is straightforward to see that $\lambda_i(\bdelta^*)=a_i-\dfrac{\sigma_{ii}}2, i=1,2$.
If $\lambda_1(\bdelta^*)<0$, (resp. $\lambda_2(\bdelta^*)<0$) there is no invariant probability measure on $\R^{\circ}_{1+}:=\{(x_1,0): x_1>0\}$ (resp. $\R^{\circ}_{2+}:=\{(0,x_2): x_2>0\}$)
in view of Theorem \ref{t:pers}.
If $\lambda_i(\bdelta^*)>0$, there is a unique invariant probability measure $\mu_i$ on $\R^{\circ}_{i+}$, $i=1,2$.
By Lemma \ref{l:int}, we have
$$\lambda_i(\mu_i)=a_i-\dfrac{\sigma_{ii}}2-b_i\int_{\R^{\circ}_{i+}}x_i\mu_i(d\bx)=0$$
which implies
\begin{equation}\label{e2-ex1}
\int_{\R^{\circ}_{i+}}x_i\mu_i(d\bx)=\dfrac{2a_i-\sigma_{ii}}{2b_i}.
\end{equation}
Thus
$$\lambda_2(\mu_1)=\int_{\R^{\circ}_{1+}}\left[a_2-\frac{\sigma_{22}}{2}-c_2x_1\right]\mu_1(d\bx)=a_2-\frac{\sigma_{22}}{2}-c_2\dfrac{2a_1-\sigma_{11}}{2b_1}$$
and
$$\lambda_1(\mu_2)=\int_{\R^{\circ}_{2+}}[a_1-\frac{\sigma_{11}}{2}-c_1x_2]\mu_2(d\bx)=a_1-\frac{\sigma_{11}}{2}-c_1\dfrac{2a_2-\sigma_{22}}{2b_2}.$$
Using Theorems \ref{t:pers} and \ref{t:ex2}, we have the following classification.
\begin{itemize}

\item If $\lambda_1(\bdelta^*)>0, \lambda_2(\bdelta^*)>0$ and $\lambda_1(\mu_2)>0, \lambda_2(\mu_1)>0$,
any invariant probability measure in $\partial\R_+^2$ has the form
$\mu=p_0\bdelta^*+p_1\mu_1+p_2\mu_2$ with $0\leq p_0,p_1,p_2$ and $p_0+p_1+p_2=1$.
It can easily be verified that
$\max_{i=1,2}\left\{\lambda_i(\mu)\right\}>0$
for any $\mu$ having the form above.
As a result there is a unique invariant probability measure $\pi^*$ on $\R^{2,\circ}_+$
and $P(t,\bx,\cdot), \bx\in\R^{2,\circ}_+$ converges to $\pi^*$ in total variation exponentially fast.
\item If $\lambda_i(\bdelta^*)<0, i=1,2$ then $X_i(t)$ converges to $(0,0)$ almost surely with the exponential rate $\lambda_i(\bdelta^*)$ for any initial condition $\bx=(x_1,x_2)\in\R^{2,\circ}_+$.
\item If $\lambda_i(\bdelta^*)>0, \lambda_j(\bdelta^*)<0$ for one $i\in\{1,2\}$ and $j\in\{1,2\}\setminus\{i\}$, then $\lambda_j(\mu_i)<0$  and $X_j(t)$ converges to $0$ almost surely with the exponential rate $\lambda_j(\mu_i)$ for any initial condition $\bx=(x_1,x_2)\in\R^{2,\circ}_+$ and the randomized occupation measure
converges weakly to $\mu_i$ almost surely.
\item If $\lambda_i(\bdelta^*)>0$, $\lambda_j(\mu_i)<0$ for both $i\in\{1,2\}, j\in\{1,2\}\setminus\{i\}$  then $p^\bx_i>0,i=1,2$ and $p^\bx_1+p^\bx_2=1$ where
$$p^\bx_i=\PP_\bx\left\{\U(\omega)=\{\mu_i\}\text{ and }\lim_{t\to\infty}\dfrac{\ln X_j(t)}t=\lambda_j(\mu_i), j\in\{1,2\}\setminus\{i\}\right\}.$$
\item If $\lambda_1(\bdelta^*)>0, \lambda_2(\bdelta^*)>0$, $\lambda_j(\mu_i)<0, \lambda_i(\mu_j)>0$ for $i,j\in\{1,2\}, i\ne j$ then $X_j(t)$ converges to $0$ almost surely with the exponential rate $\lambda_j(\mu_i)$ and the randomized occupation measure
converges weakly to $\mu_i$ almost surely for any initial condition $\bx=(x_1,x_2)\in\R^{2,\circ}_+$.
\end{itemize}
This extends and generalizes the results from \cite{EHS15}.

\end{example}
\begin{example}\label{ex2}
Consider a stochastic Lotka-Voltera model with two predators competing for one prey.
\begin{equation}\label{e1-ex2}
\begin{cases}
dX_1(t)=X_1(t)[a_1-b_1X_1(t)-c_{21}X_2(t)-c_{31}X_3(t)]dt+X_1(t)dE_1(t)\\
dX_2(t)=X_2(t)[-a_2-b_2X_2(t)+c_{12}X_1(t)-c_{32}X_3(t)]dt+X_2(t)dE_2(t)\\
dX_3(t)=X_3(t)[-a_3-b_3X_3(t)+c_{13}X_1(t)-c_{23}X_2(t)]dt+X_2(t)dE_3(t)\\
\end{cases}
\end{equation}
Assume that $a_i, b_i>0, i=1,3$, $c_{21}, c_{31}, c_{12}, c_{21}, c_{31}, c_{32}\geq0.$
Note that, if $c_{23}>0$ and $c_{32}>0$,
then \eqref{e1-ex2} describes an interacting population of  two predators $(X_2, X_3)$ competing for one prey $X_1$.
If $c_{23}<0$ and $c_{32}>0$ then
\eqref{e1-ex2} is a model of a tri-trophic food chain where $X_3$ is the top predator and $X_2$ is
the intermediate species.

In order to analyze this model,
first, consider the equation on the boundary $\{(0, x_2, x_3): x_2, x_3\geq0\}$.
Since $\lambda_i(\bdelta^*)=-a_i-\frac{\sigma_{ii}}{2}<0, i=2,3$,
an application of Theorem \ref{t:ex} to the space $\{(0, x_2, x_3): x_2, x_3\geq0\}$
shows that there is only one invariant probability measure on $\{(0, x_2, x_3): x_2, x_3\geq0\}$,
which is $\bdelta^*$.
It indicates that without the prey, both predators die out.

Now, consider the equation on the boundaries $\R_{12+}:=\{(x_1, x_2, 0): x_1, x_2\geq0\}$
and
$\R_{13+}:=\{(x_1, 0, x_3): x_1, x_3\geq0\}$.
If $\lambda_1(\bdelta^*)=a_1-\frac{\sigma_{11}}{2}<0$, $\bdelta^*$ is the unique invariant probability measure on $\R^3_+$
by virtue of Theorem \ref{t:ex}.
If $\lambda_1(\bdelta^*)>0$, there is an invariant probability measure $\mu_1$ on $\R^{\circ}_{1+}:=\{(x_1, 0, 0): x_1>0\}$.
Similarly to \eqref{e2-ex1}, we have
\begin{equation}\label{e2-ex2}
\int_{\R^{\circ}_{1+}}x_1\mu_1(d\bx)=\dfrac{2a_1-\sigma_{11}}{2b_1}.
\end{equation}
Thus
$$\lambda_i(\mu_1)=\int_{\R^{\circ}_{1+}}[-a_i-\frac{\sigma_{ii}}{2}+c_{1i}x_1]\mu_1(d\bx)=-a_i-\frac{\sigma_{ii}}{2}+c_{1i}\dfrac{2a_1-\sigma_{11}}{2b_1}, i=2,3.$$
If $\lambda_1(\bdelta^*)>0$ and $\lambda_i(\mu_1)<0, i=2,3$, by Theorem \ref{t:ex},
there is no invariant probability measure on $\R^\circ_{1i+}.$

If  $\lambda_1(\bdelta^*)>0$ and $\lambda_2(\mu_1)>0$, by Theorem \ref{t:pers},
there is an invariant probability measure $\mu_{12}$ on $\R^\circ_{12+}.$
In light of Lemma\ref{l:int},
we have
$$
\int_{\R^{\circ}_{12+}}x_1\mu_{12}(d\bx)=A_1, \int_{\R^{\circ}_{12+}}x_2\mu_{12}(d\bx)=A_2$$
where $(A_1, A_2)$ be the unique solution to
$$
\begin{cases}
a_1-\frac{\sigma_{11}}{2}-b_1x_1-c_{21}x_2=0\\
-a_2-\frac{\sigma_{22}}{2}-b_2x_2+c_{12}x_1=0.
\end{cases}
$$
In this case,
$$\lambda_3(\mu_{12})=\int_{\R^{\circ}_{12+}}\left[-a_3-\frac{\sigma_{33}}{2}+c_{13}x_1-c_{23}x_2\right]\mu_{12}(d\bx)=-a_3-\frac{\sigma_{33}}{2}+c_{13}A_1-c_{23}A_2.$$
Similarly, if $\lambda_1(\bdelta^*)>0$ and $\lambda_3(\mu_1)>0$, by Theorem \ref{t:pers},
there is an invariant probability measure $\mu_{13}$ on $\R^\circ_{13+}$
and
$$\lambda_2(\mu_{13})=\int_{\R^{\circ}_{13+}}\left[-a_2-\frac{\sigma_{22}}{2}+c_{12}x_1-c_{32}x_3\right]\mu_{13}(d\bx)=-a_2-\frac{\sigma_{22}}{2}+c_{12}\hat A_1-c_{32}\hat A_3$$
where
$(\hat A_1, \hat A_3)$ is the unique solution to
$$
\begin{cases}
a_1-\frac{\sigma_{11}}{2}-b_1x_1-c_{31}x_3=0\\
-a_3-\frac{\sigma_{33}}{2}-b_3x_3+c_{13}x_1=0.
\end{cases}
$$
By the ergodic decomposition theorem,
every invariant probability measure on $\partial \R^3_+$
is a convex combination of $\delta^*, \mu_1,\mu_{12},\mu_{13}$ (when these measures exist).
Some computations for the Lyapunov exponents
with respect to a convex combination of these ergodic measures together with an application of Theorem \ref{t:pers} show that $P(t,\bx,\cdot),\bx\in\R^{3,\circ}_+$
converges exponentially fast to an invariant probability measure $\pi^*$ on $\R^{3,\circ}_+$ if one of the following is satisfied.
\begin{itemize}
\item $\lambda_1(\bdelta^*)>0$, $\lambda_2(\mu_1)>0, \lambda_3(\mu_1)<0$ and $\lambda_3(\mu_{12})>0$.
\item $\lambda_1(\bdelta^*)>0$, $\lambda_2(\mu_1)<0, \lambda_3(\mu_1)>0$    and $\lambda_2(\mu_{13})>0$.
\item $\lambda_1(\bdelta^*)>0$, $\lambda_2(\mu_1)>0, \lambda_3(\mu_1)>0$, $\lambda_3(\mu_{12})>0$,   and $\lambda_2(\mu_{13})>0$. 
\end{itemize}

As an application of Theorem \ref{t:ex2}, we have the following classification for extinction.
\begin{itemize}
\item If $\lambda_1(\bdelta^*)<0$ then for any initial condition $\bx\in\R^{3,\circ}_+$, $X_1(t), X_2(t), X_3(t), $ converge to $0$ almost surely with the exponential rates $\lambda_i(\bdelta^*), i=1,2,3$ respectively.
\item If $\lambda_1(\bdelta^*)>0$, $\lambda_i(\mu_1)<0, i=2,3$ then $X_i(t), i=2,3$ converge to $0$ almost surely with the exponential rate $\lambda_i(\mu_1), i=2,3$ respectively and the occupation measure converges almost surely for any initial condition $\bx\in\R^{3,\circ}_+$ to $\mu_1$.
\item If $\lambda_1(\bdelta^*)>0$, $\lambda_i(\mu_1)>0, \lambda_j(\mu_{1i})<0, \lambda_j(\mu_{1})<0$  for $i,j\in\{2,3\}, i\ne j$  then $X_j(t)$ converges to  $0$ almost surely with the exponential rate $\lambda_j(\mu_{1i})$ and the occupation measure converges almost surely for any initial condition $\bx\in\R^{3,\circ}_+$ to $\mu_{1i}$.
\item  If $\lambda_1(\bdelta^*)>0$, $\lambda_2(\mu_1)>0, \lambda_3(\mu_1)>0$, $\lambda_j(\mu_{1i})<0, \lambda_i(\mu_{1j})>0$ for $i,j\in\{2,3\}, i\ne j$  then $X_j(t)$ converges to  $0$ almost surely with the exponential rate $\lambda_j(\mu_{1i})$ and the occupation measure converges almost surely for any initial condition $\bx\in\R^{3,\circ}_+$ to $\mu_{1i}$.
\item If $\lambda_1(\bdelta^*)>0$, $\lambda_2(\mu_1)>0, \lambda_3(\mu_1)>0$, $\lambda_2(\mu_{13})<0, \lambda_3(\mu_{12})<0$  then $p^\bx_i>0,i=2,3$ and $p^\bx_2+p^\bx_3=1$ where
$$p^\bx_i=\PP_\bx\left\{\U(\omega)=\{\mu_{1i}\}\,\text{ and }\,\lim_{t\to\infty}\dfrac{\ln X_i(t)}t=\lambda_j(\mu_{1i}), j\in\{2,3\}\setminus\{i\}\right\}.$$
\end{itemize}
Elementary but tedious computations show that
our results significantly improve those in \cite{LB16}.

Restricting our analysis to $\R_{12+}$ (this describes the evolution of one predator and its prey) we get
\begin{equation}\label{e3-ex2}
\begin{cases}
dX_1(t)=X_1(t)[a_1-b_1X_1(t)-c_{21}X_2(t)]dt+X_1(t)dE_1(t)\\
dX_2(t)=X_2(t)[-a_2-b_2X_2(t)+c_{12}X_1(t)]dt+X_2(t)dE_2(t).
\end{cases}
\end{equation}
In view of the analysis above,
if $\lambda_1(\bdelta^*)<0$ then $X_1(t), X_2(t)$ converge to 0 almost surely with the exponential rates $\lambda_1(\bdelta^*)$ and
$\lambda_2(\bdelta^*)=-a_2-\frac{\sigma_{22}}{2}$ respectively.
If $\lambda_1(\bdelta^*)>0$ and $\lambda_2(\mu_1)<0$ then
$X_2$ converges to $0$ almost surely with the exponential rate $\lambda_2(\mu_1)$ and the occupation measure of the process $(X_1,X_2)$ converges to $\mu_1$.
If $\lambda_1(\bdelta^*)>0,\lambda_2(\mu_1)>0$,
the transition probability of $(X_1(t), X_2(t))$ on $\R^\circ_{12+}$
converges to an invariant probability measure in total variation with an exponential rate.
These results are similar to those appearing in \cite{R03,RP07}.
However, we generalize their results by obtaining a geometric rate of convergence.

\begin{rmk}
The condition for persistence in \cite{R03,RP07}is obtained by
constructing a Lyapunov function $V$ satisfying  $\Lom V(\bx)\leq -a,\bx\in\R^{2,\circ}_+$
for some $a>0$.
The papers describe how to construct the functions $V$ rather than giving an explicit formula.
It seems to us that the function $V$ constructed in \cite{R03} is not twice differentiable.
\end{rmk}
\end{example}

\begin{example}\label{ex3}
Consider a stochastic Lotka-Voltera cooperative model
\begin{equation}\label{e1-ex4}
\begin{cases}
dX_1(t)=X_1(t)[a_1-b_1X_1(t)+c_1X_2(t)]dt+X_1(t)dE_1(t)\\
dX_2(t)=X_2(t)[a_2-b_2X_2(t)+c_2X_1(t)]dt+X_2(t)dE_2(t)
\end{cases}
\end{equation}
where $a_i,b_i,c_i > 0, i=1,2$.
As shown in Example \ref{ex1},
there exist unique invariant probability measures
$\mu_i$ on $\R_{i+}^\circ, i=1,2$ (defined in Example \ref{ex1}).
Suppose further that

$$\lambda_2(\mu_1)=\int_{\R^{\circ}_{1+}}\left[a_2-\frac{\sigma_{22}}{2}+c_2x_1\right]\mu_1(d\bx)=a_2-\frac{\sigma_{22}}{2}+c_2\dfrac{2a_1-\sigma_{11}}{2b_1}>0,$$

$$\lambda_1(\mu_2)=\int_{\R^{\circ}_{2+}}\left[a_1-\frac{\sigma_{11}}{2}+c_1x_2\right]\mu_2(d\bx)=a_2-\frac{\sigma_{22}}{2}+c_1\dfrac{2a_2-\sigma_{22}}{2b_2}>0.$$
		
and
\begin{equation}\label{e:coop}
b_1b_2-c_1c_2<0.
\end{equation}
\begin{rmk}
We note that a similar example has been studied in \cite{CM10}. The main difference is that the authors of \cite{CM10} consider demographic stochasticity instead of environmental stochasticity; their diffusion terms look like $\sqrt{X_i(t)}dE_i(t)$. In their setting the diffusion hits one of the two axes in finite time almost surely and they study the existence of quasi-stationary distributions (since there are no non-trivial stationary distributions). We note however that they still need condition \eqref{e:coop} together with some other symmetry assumptions.
\end{rmk}

Standard computations show that part (3) of Assumption \ref{a.nonde} is not satisfied by this model.
Since $a_i-\frac{\sigma_{ii}}{2}>0$ and $\lambda_i>0, i=1,2$,
Assumption \ref{a.coexn} holds.
However, we show that  the solution either blows up in finite time almost surely or there is no invariant  measure on $\R^{2,\circ}_+$.

We argue by contradiction. Suppose  $(X_1(t), X_2(t))$ does not blow up in finite time and has an invariant measure on $\R^{2,\circ}_+$. As a result $(X_1(t), X_2(t))$ is a recurrent process.
It follows from It\^o's formula that
$$
\begin{aligned}
\dfrac{b_2\ln X_1(t)+b_1\ln X_2(t)}{t}
=&\dfrac{b_2\ln X_1(0)+b_1\ln X_2(0)}{t}+b_2\left(a_1-\frac{\sigma_{11}}{2}\right)+b_1\left(a_2-\frac{\sigma_{22}}{2}\right)\\
&+(c_1c_2-b_1b_2)\dfrac1t\int_0^tX_1(s)ds+\dfrac{1}t\int_0^t(b_2dE_1(s)+b_1dE_2(s))
\end{aligned}
$$
Since
$$
\lim_{t\to\infty}\dfrac{1}t\int_0^t(b_2dE_1(s)+b_1dE_2(s))=0\text{ a.s.},
$$
$b_2\left(a_1-\frac{\sigma_{11}}{2}\right)+b_1\left(a_2-\frac{\sigma_{22}}{2}\right)>0$ and $c_1c_2-b_1b_2>0$,
it follows that
$$
\limsup_{t\to\infty}\dfrac{b_2\ln X_1(t)+b_1\ln X_2(t)}{t}>0\text{ a.s.}.
$$
Thus, $(X_1(t), X_2(t))$ cannot be a recurrent process in $\R^{2,\circ}_+$. This is a contradiction.

\end{example}
\begin{example}\label{ex4}
Consider the two dimensional system
\begin{equation}\label{e1-ex5}
dX_i(t)=X_i(t) f_i(\BX(t))dt+X_i(t)g_i(\BX(t))dE_i(t), ~i=1,2.
\end{equation}
Suppose that Assumptions \eqref{a.nonde} and \eqref{a.extn2} hold.
If $\lambda_i(\bdelta^*)=f_i(\0)-\frac{1}{2}g_i^2(\0)\sigma_{ii}>0$
then
$(\BX(t))$ has a unique invariant probability measure $\mu_i$
on $\R^\circ_{i+}$ (which is defined as in Example \ref{ex1}).
The density $p_i(\cdot)$ of $\mu_i$ can be found explicitly (in terms of integrals)
by solving the Fokker-Plank equation
$$-\dfrac{d}{du}[p_i(u)f_i(\hat u)]+\frac{\sigma_{ii}}{2} \dfrac{d^2}{du^2}[p_i(u)g_i^2(\hat u)]=0, u>0$$
where $\hat u=(u,0)$ if $i=1$ and $\hat u=(0,u)$ if $i=2$.
Then $\lambda_j(\mu_i), i,j=1,2, i\ne j$ can be computed in terms of integrals.
Using arguments similar to those in Examples \ref{ex1} and \ref{ex2},
we have the following classification,
which generalizes the Lotka-Volterra competitive and predator-prey models in previous examples.
\begin{itemize}
\item If $\lambda_i(\bdelta^*)>0,i=1,2, \lambda_1(\mu_2)>0,\lambda_2(\mu_1)>0$ then  there is a unique invariant probability measure $\pi^*$ on $\R^{2,\circ}_+$
and $P(t,\bx,\cdot), \bx\in\R^{2,\circ}_+$ converges to $\pi^*$ in total variation exponentially fast.
\item If $\lambda_i(\bdelta^*)>0,\lambda_j(\bdelta^*)<0$, $\lambda_j(\mu_i)>0$ for some $i=1,2$ and $j\ne i$, then  there is a unique invariant probability measure $\pi^*$ on $\R^{2,\circ}_+$
and $P(t,\bx,\cdot), \bx\in\R^{2,\circ}_+$ converges to $\pi^*$ in total variation exponentially fast.
\item If $\lambda_i(\bdelta^*)<0, i=1,2$ then $X_i(t)$ converges to $0$ at the exponential rate $\lambda_i(\bdelta^*), i=1,2$.
\item If $\lambda_i(\bdelta^*)>0, \lambda_j(\mu_i)<0$ and $\lambda_j(\bdelta^*)<0$ for $i,j=1,2, i\ne j$ then $X_j(t)$ converges to $0$ at the exponential rate $\lambda_j(\mu_i)$
and the randomized occupation measure converges weakly to $\mu_i$.

\item If $\lambda_i(\bdelta^*)>0, \lambda_j(\mu_i)<0$ and $\lambda_j(\bdelta^*)>0, \lambda_i(\mu_j)>0$ for $i,j=1,2, i\ne j$ then $X_j(t)$ converges to $0$ at the exponential rate $\lambda_j(\mu_i)$
and the randomized occupation measure converges weakly to $\mu_i$.
\item If $\lambda_1(\bdelta^*)>0, \lambda_2(\bdelta^*)>0$, $\lambda_1(\mu_2)<0, \lambda_2(\mu_1)<0,$ then $p^\bx_i>0,i=1,2$ and $p^\bx_1+p^\bx_2=1$ where
$$p^\bx_i=\PP_\bx\left\{\U(\omega)=\{\mu_i\}\text{ and }\lim_{t\to\infty}\dfrac{\ln X_j(t)}t=\lambda_j(\mu_i), j\in\{1,2\}\setminus\{i\}\right\}.$$
\end{itemize}
\end{example}

\begin{example}

Our methods can also be used to study the simple food chain
\begin{equation}\label{e:stoc}
\begin{split}
dX_1(t) &= X_1(t)(a_{10} - a_{11}X_1(t) - a_{12}X_2(t))\,dt + X_1(t)\,dE_1(t)\\
dX_2(t) &= X_2(t)(-a_{20} + a_{21}X_1(t) - a_{22}X_2(t) - a_{23}X_3(t))\,dt+X_2(t)\,dE_2(t)\\
&\mathrel{\makebox[\widthof{=}]{\vdots}} \\
dX_{n-1}(t) &= X_{n-1}(t)(-a_{n-1,0}+a_{n-1,n-2}X_{n-2}(t) -a_{n-1,n-1}X_{n-1}(t)- a_{n-1,n}X_n)\,dt\\
&~~~~~+X_{n-1}(t)\,dE_{n-1}(t)\\
dX_n(t) &= X_n(t)(-a_{n0} + a_{n,n-1}X_{n-1}(t)-a_{nn}X_n(t))\,dt + X_n(t)\,dE_n(t).
\end{split}
\end{equation}
In this model $X_1$ describes a prey species, which is at the bottom of the food chain. The next $n-1$ species are predators. Species $1$ has a per-capita growth rate $a_{10}>0$ and its members compete for resources according to the intracompetition rate $a_{11}>0$. Predator species $j$ has a death rate $-a_{j0}<0$, preys upon species $j-1$ at rate $a_{j,j-1}>0$, competes with its own members at rate $a_{jj}> 0$ and is preyed upon by predator $j+1$ at rate $a_{j,j+1}>0$. The last species, $X_n$, is considered to be the apex predator of the food chain.
Define the stochastic growth rate $\tilde a_{10} := a_{10}-\frac{\sigma_{11}}{2}$ and the stochastic death rates $\tilde a_{j0} := a_{j0} + \frac{\sigma_{jj}}{2}, j=1,\dots,n$. For fixed $j\in \{1,\dots,n\}$ write down the system
\begin{equation}\label{e:systemm}
\begin{split}
-a_{11}x_1 - a_{12}x_2 &= -\tilde a_{10}\\
a_{21}x_1 - a_{22}x_2 - a_{23}x_3 &=\tilde a_{20}\\
&\mathrel{\makebox[\widthof{=}]{\vdots}} \\
a_{j-1,j-2}x_{j-2} - a_{j-1,j-1}x_{j-1} -a_{j-1,j}x_{j} &= \tilde a_{j-1,0}\\
a_{j,j-1}x_{j-1} - a_{jj}x_j &= \tilde a_{j0}.
\end{split}
\end{equation}
It is easy to show that \eqref{e:systemm} has a unique solution, say $(x^{(j)}_1,\dots,x^{(j)}_j)$.
Define the invasion rate of the $j+1$st predator by
\begin{equation}\label{e:inv_j+1}
\I_{j+1} = -\tilde a_{j+1,0} + a_{j+1,j} x^{(j)}_j.
\end{equation}
Furthermore, let
$$\R_+^{(j),\circ}:=\{\bx=(x_1,\dots,x_n)\in\R^n_+: x_k>0\,\text{ for } k\leq j; x_{k}=0\,\text{ for } j<k\leq n\}.$$
\begin{thm}\label{t:mainn}
Suppose $n\geq 2$, and $\BX(0)=\bx\in\R_+^{n,\circ}$. We have the following classification.
\begin{itemize}
\item [(i)] If $\I_n>0$ then the food chain $\BX$
is strongly stochastically persistent and its transition probability converges to its unique invariant probability measure $\pi^{(n)}$ on $\R_+^{n,\circ}$ exponentially fast in total variation.
\item[(ii)]  Suppose that $\I_{j^*}>0$ and $\I_{j^*+1}<0$ for some $j^*<n$. Then
\[
\PP_x\left\{\lim_{t\to\infty}\frac{\ln X_k(t)}{t}=\tilde a_{k0}\right\}= 1, k>j^*
\]
i.e. the predators $(X_{j^*+1},\dots,X_n)$ go extinct exponentially fast.
At the same time, the normalized occupation measure of $(X_1,\dots,X_{j^*})$ converges weakly to the unique invariant probability measure $\pi^{(j^*)}$ on $\R_+^{(j^*),\circ}$.

\end{itemize}
\end{thm}
For more details regarding this model, as well as results when the noise is degenerate we refer the reader to \cite{HN17, HN17b}.

\end{example}

\section{Invariant measures, Lyapunov exponents and log-Laplace transforms}\label{s:prep}
In this section we explore some of the properties of the SDE \eqref{e:system}. These will be used in later sections in order to prove the main results. In view of \eqref{a.tight},
there is an $M>0$
such that
\begin{equation}\label{e2.2}
\left[\dfrac{\sum_i c_ix_if_i(\bx)}{1+\bc^\top\bx}-\dfrac12\dfrac{\sum_{i,j} \sigma_{ij}c_ic_jx_ix_jg_i(\bx)g_j(\bx)}{(1+\bc^\top\bx)^2}+\gamma_b\left(1+\sum_{i} (|f_i(\bx)|+g_i^2(\bx))\right)\right]<0
\end{equation}
if $\|\bx\|\geq M$.
Since
\[
|g_i(\bx)g_j(\bx)\sigma_{ij}|\leq 2|\sigma_{ij}|(|g_i(\bx)|^2+|g_j(\bx)|^2)
\]
we can find $\delta_0\in\left(0,\frac{\gamma_b}{2}\right)$ such that
\begin{equation}\label{e2.3}
3\delta_0\sum_{i,j} |g_i(\bx)g_j(\bx)\sigma_{ij}|+\delta_0\sum_ig_i^2(\bx)\leq \gamma_b\sum_ig_i^2(\bx)\,,\,\bx\in\R^n_+.
\end{equation}
By \eqref{e2.2} and \eqref{e2.3}, we have
\begin{equation}\label{e2.1}
\begin{aligned}
\dfrac{\sum_ic_ix_if_i(\bx)}{1+\bc^\top\bx}&-\dfrac{1}2\dfrac{\sum_{i,j} \sigma_{ij}c_ic_jx_ix_jg_i(\bx)g_j(\bx)}{(1+\bc^\top\bx)^2}+\gamma_b+\delta_0\sum_i(2|f_i(\bx)|+g^2_i(\bx))\\
&+3\delta_0\sum_{i,j} |g_i(\bx)g_j(\bx)\sigma_{ij}|<0\,\text{ for all }\, \|\bx\|\geq M.
\end{aligned}
\end{equation}
For $\bp=(p_1,\cdots,p_n)\in\R^{n,\circ}_+$, $\|\bp\|\leq\delta_0$,
define the function $V: \R^{n,\circ}_+\to\R_+$ by
\begin{equation}\label{e:V}
V(\bx):=\dfrac{1+\bc^\top\bx}{\prod_i x_i^{p_i}}.
\end{equation}
Using \eqref{e2.1} one can define
\begin{equation}\label{e:H}
\begin{aligned}
H:=\sup\limits_{\bx\in\R^{n}_+}\Bigg\{&\dfrac{\sum_ic_ix_if_i(\bx)}{1+\bc^\top\bx}-\dfrac{1}2\dfrac{\sum_{i,j} \sigma_{ij}c_ic_jx_ix_jg_i(\bx)g_j(\bx)}{(1+\bc^\top\bx)^2}\\
&+\gamma_b+\delta_0\sum_i(2|f_i(\bx)|+g_i^2(\bx))
+3\delta_0\sum_{i,j} |g_i(\bx)g_j(\bx)\sigma_{ij}|\Bigg\}<\infty.
\end{aligned}
\end{equation}
\begin{lm}\label{lm2.0}
For any $\bx\in\R^n_+$, there exists a pathwise unique strong solution $(\BX(t))$ to \eqref{e:system}
with initial value $\BX(0)=\bx$.
Let $I\subset\{1,\cdots,n\}$ and $\bx\in\R^{I,\circ}_+$
where $$\R^{I,\circ}_+=\left\{\bx\in \R^n_+: x_i=0 \text{ if } i\notin I \text{ and }x_i>0 \text{ if } i\in I\right\}.$$
The solution $(\BX(t))$ with initial value $\bx$ will stay forever in $\R^{I,\circ}_+$ with probability 1.
Moreover, for $\bx\in\R^{n,\circ}_+$ and $V$ defined by \eqref{e:V}, we have
\begin{equation}\label{e1-lm2.2}
\E_\bx V^{\delta_0}(\BX(t))\leq \exp(\delta_0Ht) V^{\delta_0}(\bx).
\end{equation}
\end{lm}
\begin{lm}\label{lm2.2}
There are $H_1, H_2>0$ such that for any $\bx\in\R^n_+, t>0$

\begin{equation}\label{e2-lm2.2}
\E_\bx (1+\bc^\top\BX(t))^{\delta_0}\leq H_1+(1+\bc^\top\bx)^{\delta_0}e^{-\delta_0\gamma_bt}
\end{equation}

and
\begin{equation}\label{e3-lm2.2}
\E_\bx\int_0^t  (1+\bc^\top\BX(s))^{\delta_0}\left[1+\sum_i(|f_i(\BX(s)|+|g_i(\BX(s)|^2)\right]ds\leq H_2((1+\bc^\top\bx)^{\delta_0}+t).
\end{equation}
Moreover,  the solution process $(\BX(t))$ is a Feller process on $\R^n_+$.
\end{lm}
\begin{rmk}
There are different possible definitions of ``Feller'' in the literature. What we mean by Feller is that the semigroup $(T_t)_{t\geq 0}$ of the process maps the set of bounded continuous functions $C_b(\R_+^n)$ into itself i.e.
\[
T_t(C_b(\R_+^n)) \subset C_b(\R_+^n), ~t\geq 0.
\]
\end{rmk}
Define the family of measures
$$\Pi^\bx_t(\cdot):=\dfrac1t\int_0^t\PP_\bx\{\BX(s)\in\cdot\}ds,\,\bx\in\R^n_+, t>0.$$
\begin{lm}\label{lm2.3}
Let $\mu$ be an invariant probability measure of $\BX$.
Then
$$\int_{\R^n_+}(1+\bc^\top\bx)^{\delta_0}\left(1+\sum_i (|f_i(\bx)|+|g_i(\bx)|^2)\right)\mu(d\bx)\leq H_2$$
and
$$\int_{\R^n_+}\left(\dfrac{\sum_i c_ix_if_i(\bx)}{1+\bc^\top\bx}-\dfrac{1}2\dfrac{\sum_{i,j} \sigma_{ij}c_ic_jx_ix_jg_i(\bx)g_j(\bx)}{(1+\bc^\top\bx)^2}\right)\mu(d\bx)=0.$$
\end{lm}

\begin{lm}\label{lm2.4}
Suppose the following
\begin{itemize}
\item  The sequences $(\bx_k)_{k\in N}\subset \R_+^n, (T_k)_{k\in \N}\subset \R_+$ are such that $\|\bx_k\|\leq M$, $T_k>1$ for all $k\in \N$ and $\lim_{k\to\infty}T_k=\infty$.

\item The sequence $(\Pi^{\bx_k}_{T_k})_{k\in \N}$ converges weakly to an invariant probability measure
$\pi$.

\item The function $h:\R^n_+\to\R$ is any continuous function satisfying
$|h(\bx)|<K_h(1+\bc^\top\bx)^\delta(1+\sum_i(|f_i(\bx)|+|g_i(\bx)|^2))$, $\bx\in \R^n_+$,
for some $K_h\geq 0$, $\delta<\delta_0$.
\end{itemize}
Then one has
\[\lim_{k\to\infty}\int_{\R^n_+}h(\bx)\Pi^{\bx_k}_{T_k}(d\bx)= \int_{\R^n_+}h(\bx)\pi(d\bx).\]
\end{lm}

\begin{lm}\label{lm2.5}
Let $Y$ be a random variable, $\theta_0>0$ a constant, and suppose $$\E \exp(\theta_0 Y)+\E \exp(-\theta_0 Y)\leq K_1.$$
Then the log-Laplace transform
$\phi(\theta)=\ln\E\exp(\theta Y)$
is twice differentiable on $\left[0,\frac{\theta_0}2\right)$ and
$$\dfrac{d\phi}{d\theta}(0)= \E Y,$$
$$0\leq \dfrac{d^2\phi}{d\theta^2}(\theta)\leq K_2\,, \theta\in\left[0,\frac{\theta_0}2\right)$$
 for some $K_2>0$ depending only on $K_1$.
\end{lm}

\begin{rmk}
We note that we got the very nice idea of using the log-Laplace transform in the proofs of our persistence results from the manuscript \cite{B14}.
\end{rmk}
To proceed, let us recall some technical concepts and results needed to prove the main theorem.
 Let ${\bf\Phi}=(\Phi_0,\Phi_1,\dots)$ be a discrete-time Markov chain on a general state space $(E,\mathcal{E})$, where $\mathcal{E}$ is a countably generated $\sigma$-algebra.
 Denote by $\mathcal{P}$ the Markov transition kernel for ${\bf\Phi}$.
If there is a non-trivial $\sigma$-finite positive measure $\varphi$ on $(E,\mathcal{E})$ such that for
any $A\in\mathcal{E}$ satisfying $\varphi(A)>0$ we have
$$\sum_{n=1}^\infty \mathcal{P}^n(x, A)>0,\, x\in E$$
 where  $ \mathcal{P}^n$ is the $n$-step transition kernel of ${\bf\Phi}$, then the Markov chain ${\bf\Phi}$ is called \textit{irreducible}.
It can be shown (see \cite{EN}) that if ${\bf\Phi}$ is irreducible, then there exists a positive integer $d$ and disjoint subsets $E_0,\dots,
E_{d-1}$ such that for all $i=0,\dots, d-1$ and all $x\in E_i$, we have
$$\mathcal{P}(x,E_j)=1 \text{ where } j=i+1 \text{ (mod } d).$$
The smallest positive integer $d$ satisfying the above is called the period of ${\bf\Phi}.$
An \textit{aperiodic} Markov chain is a chain with period $d=1$.

A set $C\in\mathcal{E}$ is called \textit{petite}, if there exists a non-negative sequence $(a_n)_{n\in\N}$ with $\sum_{n=1}^\infty a_n=1$
and a nontrivial positive measure $\nu$ on $(E,\mathcal{E})$
such that
$$\sum_{n=1}^\infty a_n \mathcal{P}^n(x, A)\geq\nu(A),\,\, x\in C, A\in\mathcal{E}.$$
We have the following lemma
\begin{lm}\label{lm2.7}
For any $T>0$ the Markov chain $\{(\BX(kT), k\in \N\}$ on $\R^{n,\circ}_+$ is irreducible and aperiodic.
Moreover, every compact set $K\subset\R^{n,\circ}_+$ is petite.
\end{lm}
The proofs of the above lemmas are collected in the Appendix.

\section{Persistence}\label{s:perm}
This section is devoted to finding conditions under which $\BX$ converges to a unique invariant probability measure supported on $\R^{n,\circ}_+$.

It is shown in \cite[Lemma 4]{SBA11} by the minmax principle that Assumption \ref{a.coexn} is equivalent to the existence of $\mathbf p>0$ such that
\begin{equation}\label{e.p}
\min\limits_{\mu\in\M}\left\{\sum_{i}p_i\lambda_i(\mu)\right\}:=2\rho^*>0.
\end{equation}
By rescaling if necessary, we can assume that $\|\bp\|=\delta_0$.
\begin{lm}\label{lm3.1}
Suppose that Assumption \ref{a.coexn} holds. Let $\bp$ and $\rho^*$ be as in \eqref{e.p}.
There exists a $T^*>0$ such that, for any $T>T^*$, $\bx\in\partial\R^n_+, \|\bx\|\leq M$ one has
\begin{equation}
\dfrac1T\int_0^T\E_\bx\Phi(\BX(t))dt\leq-\rho^*
\end{equation}
where
\begin{equation}
\begin{aligned}
\Phi(\bx):=\dfrac{\sum_ic_ix_if_i(\bx)}{1+\bc^\top\bx}-\dfrac{1}2\dfrac{\sum_{i,j} \sigma_{ij}c_ic_jx_ix_jg_i(\bx)g_j(\bx)}{(1+\bc^\top\bx)^2}
-\sum_{i}p_i\left(f_i(\bx)-\dfrac{\sigma_{ii}g^2_i(\bx)}{2} \right).
\end{aligned}
\end{equation}
\end{lm}
\begin{proof}
We argue by contradiction. Suppose that the conclusion of this lemma is not true.
Then, we can find $\bx_k\in\partial\R^n_+, \|\bx_k\|\leq M$
and $T_k>0$, $\lim_{k\to\infty} T_k=\infty$
such that
\begin{equation}\label{e3.9}
\dfrac1T_k\int_0^{T_k}\E_{\bx_k}\Phi(\BX(t))dt>-\rho^*\,,\,k\in\N.
\end{equation}
Note that
\[
\Pi^{\bx_k}_t(d\by):=\dfrac1t\int_0^t\PP_{\bx_k}\{\BX(s)\in d\by\}ds.
\]

By Tonelli's Theorem we get that

\begin{equation}\label{e:Ton}
\begin{aligned}
\int_{\R_+^n}(1+\bc^\top\by)^{\delta_0} \Pi^{\bx_k}_t(d\by) &= \int_{\R_+^n}(1+\bc^\top\by)^{\delta_0} \dfrac1t\int_0^t\PP_{\bx_k} \{\BX(s)\in d\by\}ds\,d\by\\
&= \dfrac1t\int_0^t \E_{\bx_k} (1+\bc^\top\BX(s))^{\delta_0}\,ds.
\end{aligned}
\end{equation}
It follows from Lemma \ref{lm2.2} that
\begin{equation}
\begin{aligned}
\sup_{k\in\N, t\geq 0} \int_{\R_+^n}(1+\bc^\top\by)^{\delta_0} \Pi^{\bx_k}_t(d\by) &= \sup_{k\in\N, t\geq 0}  \dfrac1t\int_0^t \E_{\bx_k}  (1+\bc^\top\BX(s))^{\delta_0}\,ds\\
&\leq     \sup_{\| \bx\|\leq M, t\geq 0}  \dfrac1t\int_0^t   \left(H_1+(1+\bc^\top\bx)^{\delta_0}e^{-\delta_0\gamma_bs}\right)\,ds\\
&<\infty.
\end{aligned}
\end{equation}
This implies that the family $\left(\Pi^{\bx_k}_{T_k}\right)_{k\in \N}$ is tight in $\R_+^n$. As a result
$\left(\Pi^{\bx_k}_{T_k}\right)_{k\in \N}$ has a convergent subsequence in the weak$^*$-topology.
Without loss of generality, we can suppose that $\left\{\Pi^{\bx_k}_{T_k}:k\in\N\right\}$
is a convergent sequence in the weak$^*$-topology. It can be shown (by \cite[Theorem 9.9]{EK09} or by \cite[Proposition 6.4]{EHS15}) that its limit is an invariant probability measure $\mu$ of $(\BX(t))$. As a consequence of Lemma \ref{lm2.4}
$$\lim_{k\to\infty}\dfrac1T_k\int_0^{T_k}\E_{\bx_k}\Phi(\BX(t))dt=\int_{\R^n_+}\Phi(\bx)\mu(d\bx).$$
In view of Lemma \ref{lm2.3} and \eqref{e.p} we get that
$$\lim_{k\to\infty}\dfrac1T_k\int_0^{T_k}\E_{\bx_k}\Phi(\BX(t))dt=-\sum_{i=1}^np_i\lambda_i(\mu)\leq -2\rho^*,$$
which contradicts \eqref{e3.9}.
\end{proof}

From now on let $n^*\in\N$ be such that
\begin{equation}\label{e:n*}
\gamma_b(n^*-1)>H.
\end{equation}
\begin{prop}\label{prop2.1}
Let $V(\cdot)$ be defined by \eqref{e:V} with $\bp$ and $\rho^*$ satisfying \eqref{e.p} and $T^*>0$ satisfying the assumptions of Lemma \ref{lm3.1}.
There are $\theta\in\left(0,\frac{\delta_0}2\right)$, $K_\theta>0$, such that for any $T\in[T^*,n^*T^*]$ and $\bx\in\R^{n,\circ}_+, \|\bx\|\leq M$,
$$\E_\bx V^\theta(\BX(T))\leq V^\theta(\bx)\exp\left(-\frac{1}{2}\theta \rho^*T\right) +K_\theta.$$
\end{prop}
\begin{proof}
We have from It\^o's formula that
\begin{equation}\label{e:G}
\ln V(\BX(T))=\ln V(\BX(0)) + G(T)
\end{equation}
where
\begin{equation}
\begin{aligned}
G(T)=\int_0^T\Phi(\BX(t))dt+ \int_0^T\left[\dfrac{\sum_ic_iX_i(t)g_i(\BX(t))dE_i(t)}{1+\bc^\top\BX(t)}-\sum_ip_ig_i(\BX(t))dE_i(t)\right].
\end{aligned}
\end{equation}
In view of \eqref{e:G} and \eqref{e1-lm2.2}
\begin{equation}\label{e3.4_2}
\E_\bx \exp(\delta_0 G(T))=\dfrac{\E_\bx V^{\delta_0}(\BX(T))}{V^{\delta_0}(\bx)}\leq  \exp(\delta_0 HT).
\end{equation}
Let $\hat V(\cdot):\R^{n,\circ}_+\mapsto\R_+$ be defined by $\hat V(\bx)=(1+\bc^\top\bx)\prod_{i=1}^n x_i^{p_1}$.
We can use \eqref{e2.1} and some of the estimates from the proof of Lemma \ref{lm2.0} to obtain
\begin{equation}\label{vhat-1}
\dfrac{\E_\bx \hat V^{\delta_0}(\BX(T))}{\hat V^{\delta_0}(\bx)}\leq  \exp(\delta_0 HT).
\end{equation}
Note that
\begin{equation}\label{vhat-2}
V^{-\delta_0}(\bx)=\hat V^{\delta_0}(\bx)(1+\bc^\top\bx)^{-2\delta_0}\leq \hat V^{\delta_0}(\bx).
\end{equation}
Applying \eqref{vhat-2} to \eqref{vhat-1} yields
\begin{equation}\label{e3.5}
\begin{aligned}
\E_\bx \exp(-\delta_0 G(T))=&\dfrac{\E_\bx V^{-\delta_0}(\BX(T))}{V^{-\delta_0}(\bx)}\\
\leq&\dfrac{\E_\bx\hat V^{\delta_0}(\BX(T))}{V^{-\delta_0}(\bx)}\\
\leq& \dfrac{\E_\bx\hat V^{\delta_0}(\BX(T))}{\hat V^{\delta_0}(\bx)}(1+\bc^\top\bx)^{2\delta_0}\\
\leq& (1+\bc^\top\bx)^{2\delta_0}\exp(\delta_0 HT).
\end{aligned}
\end{equation}

By \eqref{e3.4_2} and \eqref{e3.5} the assumptions of Lemma \ref{lm2.5} hold for $G(T)$. Therefore,
there is $\tilde K_2\geq 0$ such that
$$0\leq \dfrac{d^2\tilde\phi_{\bx,T}}{d\theta^2}(\theta)\leq \tilde K_2\,\text{ for all }\,\theta\in\left[0,\frac{\delta_0}2\right),\, \bx\in\R^{n,\circ}_+, \|\bx\|\leq M, T\in [T^*,n^*T^*]$$
where
$$\tilde\phi_{\bx,T}(\theta)=\ln\E_\bx \exp(\theta G(T)).$$
In view of Lemma \ref{lm3.1} and the Feller property of $(\BX(t))$,
there exists a $\tilde\delta>0$ such that
if $\|\bx\|\leq M$, $\dist(\bx,\partial\R^n_+)<\tilde\delta$ and $T\in [T^*,n^*T^*]$
then
\begin{equation}\label{e3.6}
\begin{aligned}
\E_\bx G(T)=&\int_0^T\E_\bx\left( \dfrac{\sum_ic_iX_i(t)f_i(\BX(t))}{1+\bc^\top\BX(t)}- \dfrac{\sum_{i,j} c_ic_jX_i(t)X_j(t)g_i(\BX(t))g_j(\BX(t))\sigma_{ij}}{2(1+\bc^\top\BX(t))^2}\right)dt\\
&-\sum_{i=1}^np_i\int_0^T\E_\bx\left(f_i(\BX(t))-\dfrac{\sigma_{ii}g^2_i(\BX(t))}{2} \right)dt\leq -\dfrac34\rho^*T.
\end{aligned}
\end{equation}
Another application of Lemma \ref{lm3.1} yields
$$\dfrac{d\tilde\phi_{\bx,T}}{d\theta}(0)=\E_\bx G(T)\leq -\dfrac34\rho^*T.$$
By a Taylor expansion around $\theta=0$, for $\|\bx\|\leq M, \dist(\bx,\partial\R^n_+)<\tilde\delta, T\in [T^*,n^*T^*]$ and $\theta\in\left[0,\frac{\delta_0}2\right)$ we have
$$\tilde\phi_{\bx,T}(\theta)\leq -\dfrac34\rho^*T\theta+\theta^2\tilde K_2 .$$
If we choose any $\theta\in\left(0,\frac{\delta_0}2\right)$ satisfying
$\theta<\frac{\rho^*T^*}{4\tilde K_2}$, we obtain that
\begin{equation}\label{e3.10}
\tilde\phi_{\bx,T}(\theta)\leq -\dfrac12\rho^*T\theta\,\,\text{ for all }\,\bx\in\R^{n,\circ},\|\bx\|\leq M, \dist(\bx,\partial\R^n_+)<\tilde\delta, T\in [T^*,n^*T^*].
\end{equation}
In light of \eqref{e3.10}, we have for such $\theta$ and $\|\bx\|\leq M, 0<\dist(\bx,\partial\R^n_+)<\tilde\delta, T\in [T^*,n^*T^*]$ that
\begin{equation}\label{e3.11}
\dfrac{\E_\bx V^\theta(\BX(T))}{V^\theta(\bx)}=\exp \tilde\phi_{\bx,T}(\theta)\leq\exp\left(-\frac{1}{2}\rho^*T\theta\right).
\end{equation}
In view of \eqref{e1-lm2.2},
we have for $\bx$ satisfying $\|\bx\|\leq M, \dist(\bx,\partial\R^n_+)\geq\tilde\delta$ and $T\in  [T^*,n^*T^*]$ that
\begin{equation}\label{e3.12}
\E_\bx V^\theta(\BX(T))\leq \exp(\theta n^*T^*H)\sup\limits_{\|\bx\|\leq M, \dist(\bx,\partial\R^n_+)\geq\tilde\delta}\{V(\bx)\}=:K_\theta<\infty.
\end{equation}
The desired result follows from \eqref{e3.11} and \eqref{e3.12}.
\end{proof}
\begin{thm}\label{thm3.1}
Suppose that Assumptions \ref{a.nonde} and \ref{a.coexn} hold.
Let $\theta$ be as in Proposition \ref{prop2.1}, $n^*$ as in \eqref{e:n*}.
There are $\kappa=\kappa(\theta,T^*)\in(0,1)$, $\tilde K=\tilde K(\theta,T^*)>0$   such that
\begin{equation}\label{e:lya}
\E_\bx V^\theta(\BX(n^*T^*))\leq \kappa V^\theta(x)+\tilde K\,\text{ for all }\, \bx\in\R^{n,\circ}_+.
\end{equation}
As a result,
$\BX$ is strongly persistent. Furthermore, the convergence of its transition probability in total variation to its unique probability measure $\pi^*$ on $\R^{n,\circ}_+$ is
exponentially fast. For any initial value $\mathbf{x}\in\R^{n,\circ}_+$ and any $\pi^*$-integrable function $f$ we have
\begin{equation}\label{slln}
\PP_\bx\left\{\lim\limits_{T\to\infty}\dfrac1T\int_0^Tf\left(\BX^{}(t)\right)dt=\int_{\R_+^{n,\circ}}f(\mathbf{u})\pi^*(d\mathbf{u})\right\}=1.
\end{equation}

\end{thm}
\begin{proof}
By direct calculation and using \eqref{e2.1}, we have
\begin{equation}\label{et1.1}
\Lom V^\theta(\bx)\leq -\theta\gamma_bV^\theta(\bx) \text{ if } \|x\|\geq M.
\end{equation}
Define
\begin{equation}\label{e:tau}
\tau=\inf\{t\geq0: \|\BX(t)\|\leq M\}.
\end{equation}
In view of \eqref{et1.1}, we can obtain from Dynkin's formula that
$$
\begin{aligned}
\E_\bx&\left[ \exp\left(\theta\gamma_b(\tau\wedge n^*T^*)\right)V^\theta(\BX(\tau\wedge n^*T^*))\right]\\
&\leq V^\theta(\bx) +\E_\bx \int_0^{\tau\wedge n^*T^*}\exp(\theta\gamma_b s)[\Lom V^\theta(\BX(s))+ \theta\gamma_bV^\theta(\BX(s))]ds\\
&\leq V^\theta(\bx).
\end{aligned}
$$
Thus,
\begin{equation}\label{et1.2}
\begin{aligned}
V^\theta(\bx)\geq&
\E_\bx\left[ \exp\left(\theta\gamma_b(\tau\wedge n^*T^*)\right)V^\theta(\BX(\tau\wedge n^*T^*))\right]\\
=&
 \E_\bx \left[\1_{\{\tau\leq (n^*-1)T^*\}}\exp\left(\theta\gamma_b(\tau\wedge n^*T^*)\right)V^\theta(\BX(\tau\wedge n^*T^*))\right]\\
 &+\E_\bx \left[\1_{\{ (n^*-1)T^*<\tau<n^*T^*\}}\exp\left(\theta\gamma_b(\tau\wedge n^*T^*)\right)V^\theta(\BX(\tau\wedge n^*T^*))\right]\\
&+ \E_\bx \left[\1_{\{\tau\geq n^*T^*\}}\exp\left(\theta\gamma_b(\tau\wedge n^*T^*)\right)V^\theta(\BX(\tau\wedge n^*T^*))\right]\\
\geq&
 \E_\bx \left[\1_{\{\tau\leq (n^*-1)T^*\}}V^\theta(\BX(\tau))\right]\\
 &+\exp\left(\theta\gamma_b (n^*-1)T^*\right)\E_\bx \left[\1_{\{ (n^*-1)T^*<\tau<n^*T^*\}}V^\theta(\BX(\tau))\right]\\
&+\exp\left(\theta\gamma_b n^*T^*\right) \E_\bx \left[\1_{\{\tau\geq n^*T^*\}}V^\theta(\BX(n^*T^*))\right].\\
 \end{aligned}
\end{equation}

By the strong Markov property of $(\BX(t))$ and
Proposition \ref{prop2.1}, we obtain
\begin{equation}\label{et1.3}
\begin{aligned}
\E_\bx&\left[ \1_{\{\tau\leq (n^*-1)T^*\}}V^\theta(\BX(n^*T^*))\right]\\
&\leq
 \E_\bx \left[\1_{\{\tau\leq (n^*-1)T^*\}}\big[K_\theta+e^{-\frac{1}{2}\theta p^*(n^*T^*-\tau)}V^\theta(\BX(\tau))\big]\right]\\
 &\leq K_\theta+ \exp\left(-\frac{1}{2}\theta \rho^*T^*\right)\E_\bx\left[\1_{\{\tau\leq (n^*-1)T^*\}}V^\theta(\BX(\tau))\right]
 \end{aligned}
\end{equation}
By making use again of the strong Markov property of $(\BX(t))$ and
\eqref{e1-lm2.2}, we get
\begin{equation}\label{et1.4}
\begin{aligned}
\E_\bx&\left[ \1_{\{(n^*-1)T^*<\tau<n^*T^*\}}V^\theta(\BX(n^*T^*))\right]\\
&\leq
 \E_\bx \left[\1_{\{(n^*-1)T^*<\tau<n^*T^*\}}e^{\theta H(n^*T^*-\tau)}V^\theta(\BX(\tau))\right]\\
 &\leq \exp(\theta HT^*)\E_\bx\left[\1_{\{(n^*-1)T^*<\tau<n^*T^*\}}\left(V^\theta(\BX(\tau))\right)\right].
 \end{aligned}
\end{equation}
Applying \eqref{et1.3} and \eqref{et1.4} to \eqref{et1.2} yields
\begin{equation}\label{et1.5}
\begin{aligned}
V^\theta(x)
\geq&
 \E_\bx \left[\1_{\{\tau\leq (n^*-1)T^*\}}V^\theta(\BX(\tau))\right]\\
 &+\exp\left(\theta\gamma_b (n^*-1)T^*\right)\E_\bx \left[\1_{\{ (n^*-1)T^*<\tau<n^*T^*\}}V^\theta(\BX(\tau))\right]\\
&+\exp\left(\theta\gamma_b n^*T^*\right) \E_\bx \left[\1_{\{\tau\geq n^*T^*\}}V^\theta(\BX(n^*T^*))\right]\\
\geq& \exp\left(\frac{1}{2}\theta \rho^*T^*\right)\E_\bx\left[\1_{\{\tau\leq (n^*-1)T^*\}}V^\theta(\BX(n^*T^*))\right]-\exp\left(\frac{1}{2}\theta \rho^*T^*\right)K_\theta\\
 &+\exp(-\theta HT^*)\exp\left(\theta\gamma_b (n^*-1)T^*\right)\E_\bx \left[\1_{\{ (n^*-1)T^*<\tau<n^*T^*\}}V^\theta(\BX(n^*T^*))\right]\\
&+\exp\left(\theta\gamma_b n^*T^*\right) \E_\bx \left[\1_{\{\tau\geq n^*T^*\}}V^\theta(\BX(n^*T^*))\right]\\
\geq & \exp( m\theta T^*)\E_\bx V^\theta(\BX(n^*T^*))-K_\theta\exp\left(\frac{1}{2}\theta \rho^*T^*\right)
 \end{aligned}
\end{equation}
where $m=\min\left\{\frac{1}{2} \rho^*, \gamma_b n^*, \gamma_b (n^*-1)- H\right\}>0$ by \eqref{e:n*}.
The proof of \eqref{e:lya} is complete by taking $\kappa=\exp(-m\theta T^*)$ and
\[
\tilde K=K_\theta\exp\left(\frac{1}{2}\theta \rho^*T^*\right)\exp(-m\theta T^*).
\]
By Lemma \ref{lm2.7}, the Markov chain $\{\BX(kn^*T^*):k\in\N\}$ is
irreducible and aperiodic. Moreover, each compact subset of $\R^{n,\circ}_+$ is petite.
Applying the second corollary of \cite[Theorem 6.2]{MT},
we deduce from \eqref{et1.5} that
\begin{equation}\label{e.cxr}
\|P(kn^*T^*,\bx,\cdot)-\pi^*(\cdot)\|_{TV}\leq C_\bx r^k
\end{equation}
  where $\pi^*$ is an invariant probability measure of $\{\BX(kn^*T^*), k\in\N\}$ on $\R^{n,\circ}_+$, for some $r\in(0,1)$ and $C_\bx>0$ a constant depending on $\bx\in\R^{n,\circ}_+$.

On the other hand, it follows from \eqref{et1.5} and \cite[Theorem 6.2]{MT},
that for any compact set $K\subset\R^{n,\circ}_+$,
we have $\E_\bx\tau_K^*<\infty$ where $\tau_K^*$ is the first time the Markov chain $\{\BX(kn^*T^*), k\in\N\}$ enters $K$.
Thus, the process $\BX$ is a positive recurrent diffusion,
or equivalently, $\BX$  has a unique invariant probability measure on $\R^{n,\circ}_+$
(see e.g. \cite[Chapter 4]{RK}).
Because of \eqref{e.cxr}, the unique invariant probability measure of the process $\BX$
must be $\pi^*$.
Moreover, it is well-known that $\|P(t,\bx,\cdot)-\pi^*(\cdot)\|_{TV}$ is decreasing in $t$
(it can be shown easily using the Kolmogorov-Chapman equation).
We therefore obtain an exponential upper bound for $\|P(t,\bx,\cdot)-\pi^*(\cdot)\|_{TV}$.

\end{proof}
\section{Extinction}\label{s:extin}
This section is devoted to the study of conditions under which some of the species will go extinct with strictly positive probability.

\begin{lm}\label{lm4.1}
For any $\mu\in\M$ and $i\in I_\mu$ we have
$\lambda_i(\mu)=0.$
\end{lm}
\begin{proof}
In view of It\^o's formula,
$$
\dfrac{\ln X_i(t)}t=\dfrac{\ln X_i(0)}t+\dfrac1t\int_0^t\left[f_i(\BX(s))-\dfrac{g_i^2(\BX(s)\sigma_{ii}}2\right]ds+\dfrac1t\int_0^t g_i(\BX(s))dE_i(s).
$$
In the same manner as in the second part of the proof of Lemma \ref{lm2.3},
we can show that if $\BX(0)=\bx_0\in\R_+^{\mu,\circ}$ and $i\in I_\mu$,
then
$$
\lim_{t\to\infty}\dfrac1t\int_0^t\left[f_i(\BX(s))-\dfrac{g_i^2(\BX(s)\sigma_{ii}}2\right]ds=\lambda_i(\mu)~~~\PP_{\bx_0}\text{-  a.s.}
$$
and
$$\lim_{t\to\infty}\dfrac1t\int_0^tg_i(\BX(s))dE_i(s)=0~~~\PP_{\bx_0}\text{-  a.s.}
$$
On the other hand, $X_i(t), i\in I_\mu$ can go to neither $0$ nor $\infty$ as $t\to\infty$. Thus
\[
\lim_{t\to\infty}\dfrac{\ln X_i(t)}t = 0, ~~~\PP_{\bx_0}\text{-  a.s.}, i \in I_\mu
\]
which implies the desired result.
\end{proof}

Condition \eqref{ae3.2} is equivalent to the existence of
$0<\hat p_i<\delta_0, i\in I_\mu$
such that for any $\nu\in\M_\mu$ , we have
$$\sum_{i\in I_\mu}\hat p_i\lambda_i(\nu)>0.$$
Thus, there is $\check p\in (0,\delta_0)$ sufficiently small such that
\begin{equation}\label{e3.2}
\begin{aligned}
\sum_{i\in I_\mu}\hat p_i\lambda_i(\nu)-\check p\max_{i\in I_\mu^c}\{\lambda_i(\nu)\}>0 \text{ for any }\nu\in\M_\mu
\end{aligned}
\end{equation}
In view of \eqref{e3.2}, \eqref{ae3.1} and Lemma \ref{lm4.1}, there is $\rho_e>0$ such that for any $\nu\in\M_\mu\cup\{\mu\}$,
\begin{equation}\label{e3.3}
\sum_{i\in I_\mu}\hat p_i\lambda_i(\nu)-\check p\max_{i\in I_\mu^c}\left\{\lambda_i(\nu)\right\}>3\rho_e.
\end{equation}
\begin{lm}\label{lm4.2}
Suppose that Assumption \ref{a.extn} holds.
Let $M$ be as in \eqref{e2.2}, $H$ as in \eqref{e:H} and $\hat p_i, \check p, \rho_e$ as in \eqref{e3.3}. Let $n_e\in\N$ such that $\gamma_b(n_e-1)>H$.
There are $T_e\geq 0$, $\delta_e>0$ such that, for any $T\in[T_e,n_eT_e]$, $\|\bx\|\leq M, x_i<\delta_e, i\in I_\mu^c$, we have
\begin{equation}\label{e3.4}
\begin{aligned}
\dfrac1T&\int_0^T\E_\bx\left( \dfrac{\sum_ic_iX_i(t)f_i(\BX(t))}{1+\bc^\top\BX(t)}- \dfrac{\sum_{i,j} c_ic_jX_i(t)X_j(t)g_i(\BX(t))g_j(\BX(t))\sigma_{ij}}{2(1+\bc^\top\BX(t))^2}\right)dt\\
&-\sum_{i\in I_\mu}\hat p_i\dfrac1T\int_0^T\E_\bx\left(f_i(\BX(t))-\dfrac{\sigma_{ii}g^2_i(\BX(t))}{2} \right)dt\\
&+\check p\max_{i\in I_\mu^c}\left\{\dfrac1T\int_0^T\E_\bx\left(f_i(\BX(t))-\dfrac{\sigma_{ii}g^2_i(\BX(t))}{2} \right)dt\right\}
\leq -\rho_e.
\end{aligned}
\end{equation}
\end{lm}
\begin{proof}
Analogous to Lemma \ref{lm3.1}, using \eqref{e3.3}, one can show there exists a $T_e>0$ such that for any $T>T_e, \bx\in\R_+^\mu, \|\bx\|\leq M$, we have
\begin{equation}
\begin{aligned}
\dfrac1T&\int_0^T\E_\bx\left( \dfrac{\sum_ic_iX_i(t)f_i(\BX(t))}{1+\bc^\top\BX(t)}- \dfrac{\sum_{i,j} c_ic_jX_i(t)X_j(t)g_i(\BX(t))g_j(\BX(t))\sigma_{ij}}{2(1+\bc^\top\BX(t))^2}\right)dt\\
&-\sum_{i\in I_\mu}\hat p_i\dfrac1T\int_0^T\E_\bx\left(f_i(\BX(t))-\dfrac{\sigma_{ii}g^2_i(\BX(t))}{2} \right)dt\\
&+\check p\max_{i\in I_\mu^c}\left\{\dfrac1T\int_0^T\E_\bx\left(f_i(\BX(t))-\dfrac{\sigma_{ii}g^2_i(\BX(t))}{2} \right)dt\right\}
\leq -2\rho_e
\end{aligned}
\end{equation}
By the Feller property of $(\BX(t))$ and compactness of the set $\{\bx\in\R_+^\mu, \|\bx\|\leq M\}$, there is a $\delta_e>0$ such that
for any $T\in[T_e,n_eT_e]$, $\|\bx\|\leq M, x_i<\delta_e, i\in I_\mu^c$ , the estimate
\eqref{e3.4} holds.
\end{proof}
\begin{prop}\label{prop4.1}
Suppose that Assumption \ref{a.extn} holds.
Let $\delta_0>0$ be as in \eqref{e2.3}.
There is a $\theta\in(0,\delta_0)$ such that for any $T\in[T_e,n_eT_e]$ and $\bx\in\R^{n,\circ}_+$ satisfying $ \|\bx\|\leq M,$ $x_i<\delta_e,$ $i\in I_\mu^c$ one has
$$\E_\bx U_\theta(\BX(T))\leq \exp\left(-\frac{1}{2}\theta \rho_eT\right) U_\theta(\bx)$$
where
$M, T_e, \hat p_i,\check p, \delta_e, n_e$ are as in Lemma \ref{lm4.2} and
$$U_\theta(\bx)=\sum_{i\in I_\mu^c}\left[(1+\bc^\top\bx)\dfrac{x_i^{\check p}}{\prod_{j\in I_\mu} x_j^{\hat p_j}}\right]^\theta, \bx\in\R^{n,\circ}_+.$$
\end{prop}
\begin{proof}
For $i\in I_\mu^c$, let
$U_i(\bx)=(1+\bc^\top\bx)\dfrac{x_i^{\check p}}{\prod_{j\in I_\mu} x_j^{\hat p_j}}.$
Similarly to Proposition \ref{prop2.1}, by making use of Lemma \ref{lm4.2}, one can find a $\theta>0$ such that for $T\in[T_e,n_eT_e]$, $\bx\in\R^{n,\circ}_+$ with $\|\bx\|\leq M,$ and $x_i<\delta_e$ we have
$$\E U_i^\theta(\BX(T))\leq \exp\left(-\frac{1}{2}\theta \rho_eT\right) U_i^\theta(\bx).$$
The proof is complete by noting that
$$U_\theta(\bx)=\sum_{i\in I_\mu^c}U_i^\theta(\bx).$$
\end{proof}

\begin{lm}\label{lm3.3}
Let $H$ be defined by \eqref{e:H}. For $\theta\in[0,\delta_0]$ we have
$$\E_\bx U_\theta(\BX(t))\leq \exp(\theta Ht) U_\theta(\bx)\,,\, \bx\in\R^{n,\circ}_+.$$
\end{lm}
\begin{proof}
By the arguments from the proof of \eqref{e1-lm2.2}, for $\theta\leq\delta_0$, $i\in I_\mu^c$
we have
$$\E_\bx U_i^{\theta}(\BX(t))\leq \exp(\theta Ht) U_i^{\theta}(\bx)\,, \bx\in\R^{n,\circ}_+.$$
From this estimate, we can take the sum over $I_\mu^c$ to obtain the desired result.
\end{proof}
\begin{rmk}
It is key to note that the inequalities \eqref{e2.4} and \eqref{e2.5} hold if $|p_i|<\delta_0$ no matter if the $p_i$'s are negative or positive. This then allows us to have the same kind of estimates for $U_\theta$ and $V_\theta$.
\end{rmk}
\begin{thm}\label{thm4.1}
Under Assumptions \ref{a.nonde} and  \ref{a.extn} for any $\delta<\delta_0$
and any $\bx\in\R^{n,\circ}_+$ we have
\begin{equation}\label{e.extinction}
\lim_{t\to\infty}\E_\bx \bigwedge_{i=1}^n X_i^{\delta}(t)=0,
\end{equation}
where $\bigwedge_{i=1}^n a_i=\min_{i=1,\dots,n}\{a_i\}.$

\end{thm}
\begin{proof}
Just as in \eqref{et1.1}, we have 
\begin{equation}\label{et3.1}
\Lom U_\theta(\bx)\leq -\theta\gamma_bU_\theta(\bx) \text{ if } \|x\|\geq M.
\end{equation}
Let
$$\varsigma:=\dfrac{\delta_e^{\check p\theta}}{C_U^\theta},$$
$$C_U:=\sup\left\{\dfrac{\prod_{i\in I_\mu} x_i^{\hat p_i}}{1+\bc^\top\bx}: \bx\in\R^{n,\circ}_+\right\}<\infty,$$
and
$$\xi:=\inf\left\{t\geq0: U^\theta(\BX(t))\geq \varsigma\right\}.$$
Clearly, if $U_\theta(\bx)<\varsigma$, then $\xi>0$ and
for any $i\in I_\mu^c$, we get
\begin{equation}\label{e:ine}
X_i(t)\leq \delta_e\,, t\in [0,\xi).
\end{equation}
Let $$\tilde U_\theta(\bx):=\varsigma\wedge U_\theta(\bx).$$
We have from the concavity of $x\mapsto x\wedge \varsigma$ that
$$\E_\bx \tilde U_\theta(\BX(T))\leq\varsigma\wedge \E U_\theta(\BX(T)).$$
Let $\tau$ be defined by \eqref{e:tau}. By \eqref{et3.1} and Dynkin's formula, we have that
$$
\begin{aligned}
\E_\bx&\left[ \exp\left(\theta\gamma_b(\tau\wedge\xi\wedge n_eT_e)\right)U_\theta(\BX(\theta\gamma_b(\tau\wedge\xi\wedge n_eT_e))\right]\\
&\leq U_\theta(\bx) +\E_\bx \int_0^{\theta\gamma_b(\tau\wedge\xi\wedge n_eT_e)}\exp(\theta\gamma_b s)[\Lom U_\theta(\BX(s))+ \theta\gamma_bU_\theta(\BX(s))]ds\\
&\leq U_\theta(\bx).
\end{aligned}
$$
As a result,
\begin{equation}\label{et3.3}
\begin{aligned}
U_\theta(x)\geq&
\E_\bx\left[ \exp\left(\theta\gamma_b(\tau\wedge\xi\wedge n_eT_e)\right)U_\theta(\BX(\tau\wedge\xi\wedge n_eT_e))\right]\\
\geq& \E_\bx \left[\1_{\{\tau\wedge\xi\wedge(n_e-1)T_e=\tau\}}U_\theta(\BX(\tau))\right]\\
&+ \E_\bx \left[\1_{\{\tau\wedge\xi\wedge(n_e-1)T_e=\xi\}}U_\theta(\BX(\xi))\right]\\
 &+\exp\left(\theta\gamma_b (n_e-1)T_e\right) \E_\bx \left[\1_{\{(n_e-1)T_e<\tau\wedge\xi<n_eT_e\}}U_\theta(\BX(\tau\wedge\xi))\right]\\
&+\exp\left(\theta\gamma_b n_eT_e\right) \E_\bx \left[\1_{\{\tau\wedge\xi\geq n_eT_e\}}U_\theta(\BX(n_eT_e))\right].\\
 \end{aligned}
\end{equation}

By the strong Markov property of $(\BX(t))$ and
Proposition \ref{prop4.1} (which we can use because of \eqref{e:ine})
\begin{equation}\label{et3.4}
\begin{aligned}
\E_\bx&\left[ \1_{\{\tau\wedge\xi\wedge(n_e-1)T_e=\tau\}}U_\theta(\BX(n_eT_e))\right]\\
&\leq
 \E_\bx \left[\1_{\{\tau\wedge\xi\wedge(n_e-1)T_e=\tau\}}\exp\left(-\frac{1}{2}\theta \rho_e(n_eT_e-\tau)\right)U_\theta(\BX(\tau))\right] \\
 &\leq 
 \E_\bx\left[\1_{\{\tau\wedge\xi\wedge(n_e-1)T_e=\tau\}}U_\theta(\BX(\tau))\right]. \
 \end{aligned}
\end{equation}
Similarly, by the strong Markov property of $(\BX(t))$ and
Lemma \ref{lm3.3}, we obtain
\begin{equation}\label{et3.5}
\begin{aligned}
\E_\bx&\left[ \1_{\{(n_e-1)T_e<\tau\wedge\xi<n_eT_e\}}U_\theta(\BX(n_eT_e))\right]\\
&\leq
 \E_\bx \left[\1_{\{(n_e-1)T_e<\tau\wedge\xi<n_eT_e\}}\exp\left(\theta H(n_eT_e-\tau\wedge\xi)\right)U_\theta(\BX(\tau\wedge\xi ))\right] \\
 &\leq \exp(\theta HT_e)\E_\bx\left[\1_{\{(n_e-1)T_e<\tau\wedge\xi<n_eT_e\}}U_\theta(\BX( \tau\wedge\xi ))\right].
 \end{aligned}
\end{equation}
If $U_\theta(\bx)<\varsigma$ then applying \eqref{et3.4}, \eqref{et3.5} and the inequality $\tilde U_\theta(\BX(n_eT_e))\leq U_\theta(\BX(n_eT_e\wedge\xi))$  yields
\begin{equation}\label{et3.8}
\begin{aligned}
\tilde U_\theta(\bx)=U_\theta(\bx)
\geq& \E_\bx \left[\1_{\{\tau\wedge\xi\wedge(n_e-1)T_e=\tau\}}U_\theta(\BX(\tau))\right]\\
&+ \E_\bx \left[\1_{\{\tau\wedge\xi\wedge(n_e-1)T_e=\xi\}}U_\theta(\BX(\xi))\right]\\
 &+\exp\left(\theta\gamma_b (n_e-1)T_e\right) \E_\bx \left[\1_{\{(n_e-1)T_e<\tau\wedge\xi<n_eT\}}U_\theta(\BX(\tau\wedge\xi))\right]\\
&+\exp\left(\theta\gamma_b n_eT_e\right) \E_\bx \left[\1_{\{\tau\wedge\xi\geq n_eT_e\}}U_\theta(\BX(n_eT_e))\right]\\
\geq&\E_\bx \left[\1_{\{\tau\wedge\xi\wedge(n_e-1)T_e=\tau\}}U_\theta(\BX(n_eT_e))\right]\\
&+ \E_\bx \left[\1_{\{\tau\wedge\xi\wedge(n_e-1)T_e=\xi\}}\tilde U_\theta(\BX(n_eT_e))\right]\\
 &+\exp\left(\theta\gamma_b (n_e-1)T_e-\theta HT_e\right)\E_\bx \left[\1_{\{(n_e-1)T_e<\tau\wedge\xi<n_eT_e\}}U_\theta(\BX(n_eT_e))\right]\\
&+\exp\left(\theta\gamma_b n_eT_e\right) \E_\bx \left[\1_{\{\tau\wedge\xi\geq n_eT_e\}}U_\theta(\BX(n_eT_e))\right]\\
\geq& \E_\bx \tilde U_\theta(\BX(n_eT_e)) \,\quad\text{ (since } \tilde U_\theta(\cdot)\leq U_\theta(\cdot))
 \end{aligned}
\end{equation}
Clearly, if $U_\theta(\bx)\geq\varsigma$ then
\begin{equation}\label{et3.8a}
\E_\bx \tilde U_\theta(\BX(n_eT_e)) \leq \varsigma=\tilde U_\theta(\bx).
\end{equation}
As a result of \eqref{et3.8}, \eqref{et3.8a} and the Markov property of $(\BX(t))$, the sequence
$\{Y(k): k\in\N\}$ where $Y(k):=\tilde U_\theta(\BX(kn_eT_e))$ is a supermartingale.
For $\lambda\leq\varsigma$, let $\zeta_\lambda:=\inf\{k\in\N: Y(k)\geq \lambda\}$.
If $U_\theta(\bx)\leq \lambda\eps$ we have
\begin{equation}\label{e:EY_ineq}
\E_\bx  Y(k\wedge\zeta_\lambda)\leq\E_\bx Y(0)=U_\theta(\bx)\leq \lambda\eps\,\text{ for all }\, k\in\N.
\end{equation}
By assumption $\lambda\leq \varsigma$ and $Y(k)\leq \varsigma$ for any $k$. As a result  \eqref{e:EY_ineq} combined with the Markov inequality yields
$$\PP_\bx\{\zeta_\lambda<k\}\leq \lambda^{-1}\E_\bx Y(k\wedge\zeta_\lambda)\leq \eps.
$$
Next, let $k\to\infty$ to get
\begin{equation}\label{et3.7}
\PP_\bx\{\zeta_\lambda<\infty\}\leq \eps.
\end{equation}
Note that for a given compact set $\K\subset\R^{n,\circ}_+$ with nonempty interior, and for any $\eps>0$
there exists $\lambda>0$ such that
\begin{equation}\label{et3.9}
\PP_\bx\{X_i(t)\geq\lambda\,\text{ for all }\,t\in[0,n_eT_e], i=1,\dots,n\}>1-\eps,\, \bx\in\K.
\end{equation}
This standard fact can be shown in the same manner as \eqref{e2-lm4.7}, which is proved later in Lemma \ref{lm4.7}.

We show by contradiction that $(\BX(t))$ is transient.
If the process $(\BX(t))$ is recurrent in $\R^{n,\circ}_+$,
then $\BX(t)$ will enter $\K$ in a finite time almost surely given that $\BX(0)\in\R^{n,\circ}_+$.
By the strong Markov property and \eqref{et3.9}, we have
\begin{equation}\label{et3.10}
\PP_\bx\{X_i(kn_eT_e)\geq\lambda, \,i=1,\dots,n\,\text{ for some } k\in\N\}>1-\eps\,,\bx\in\R^{n,\circ}_+.
\end{equation}
If $\bx\in\R^{n,\circ}_+$ is such that $U_\theta(\bx)$ is sufficiently small then both \eqref{et3.7} and \eqref{et3.10} hold, a contradiction. Thus, $\BX$ is transient.

As a result, any weak$^*$-limit of
$P(t,\bx,\cdot)$ is a probability measure concentrated on $\partial\R^n_+$.
Similar computations to the ones from Lemma \ref{lm2.4} show that if
$P(t_k,\bx_0,\cdot)$ with $\lim_{k\to\infty}t_k=\infty$ converges weakly to $\pi$, and $h(\cdot)$ is a continuous function on $\R^n_+$ such that for all $\bx\in\R_+^n$ we have
$|h(\bx)|<K(1+\|\bx\|)^{\delta},\delta<\delta_0$
then $h(\cdot)$ is $\pi$-integrable and
$\int_{\R^n_+}h(\bx)P(t_k,\bx_0, d\bx)\to \int_{\R^n_+}h(\bx)\pi(d\bx).$

For any
$\pi$ with $\suppo(\pi)\subset\partial\R^n_+$,
we have
$$\int_{\R^n_+}\left(\bigwedge_{i=1}^n x_i^{\delta}\right)\pi(d\bx)=0,$$
and

\[
 \left(\bigwedge_{i=1}^n x_i^{\delta}\right)\leq K(1+\|\bx\|)^{\delta}.
\]
These facts imply
$$\lim_{t\to\infty}\int_{\R^n_+}\left(\bigwedge_{i=1}^n x_i^{\delta}\right)P(t_k,\bx_0, d\bx)=0$$
as desired.
\end{proof}
We also need the following lemmas.
\begin{lm}\label{lm4.4}
Suppose that Assumption \ref{a.extn2} is satisfied.
Then there is $\hat K>0$ such that
$$
\PP_\bx\left\{\limsup_{t\to\infty}\dfrac1t\int_0^t(1+\bc^\top\BX(s))^{\delta_1}\left(1+\sum_i(|f_i(\BX(s))|+|g_i(\BX(s))|^2)\right)ds\leq \hat K\right\}=1\,, \,\bx\in\R^n_+.
$$
Moreover,
\begin{equation}\label{e8-lm4.9}
\PP_{\bx}\left\{\lim_{t\to\infty}\dfrac1t\int_0^t\dfrac{\sum_ic_iX_i(s)g_i(\BX(s))}{1+\bc^\top\BX(s)}dE_i(s)=0\right\}=1\,, \,\bx\in\R^n_+.
\end{equation}

\end{lm}
\begin{lm}\label{lm4.7}
Let Assumption \ref{a.extn2} be satisfied.
There is $\hat K_1>1$ such that
\begin{equation}\label{e1-lm4.7}
\PP_\bx\left\{\liminf_{t\to\infty} \dfrac{1}t\int_0^t\1_{\{\|\BX(s)\|\leq \hat K_1\}}ds\geq\dfrac12\right\}=1,\,\bx\in\R^n_+.
\end{equation}
Moreover, for any $\eps_1,\eps_2>0$, there is a $\beta>0$ such that
for each $i=1,\cdots,n$,
\begin{equation}\label{e2-lm4.7}
\PP_\bx\{X_i(t)>\beta\,,\forall\, t\in[0,n_eT_e]\}>1-\eps_1\,\text{ if }\, \bx\in\R^n_+,\|\bx\|\leq \hat K_1, x_i>\eps_2.
\end{equation}
\end{lm}

\begin{lm}\label{lm4.5}
Let Assumption \ref{a.extn2} be satisfied.
Suppose we have a sample path  of $\BX$ satisfying
$$\limsup_{t\to\infty}\dfrac1t\int_0^t(1+\bc^\top\BX(s))^{\delta_1}\left(1+\sum_i(|f_i(\BX(s))|+|g_i(\BX(s))|^2)\right)ds\leq \hat K$$
and that there a sequence $(T_k)_{k\in\N}\subset\R^n_+$ such that
$\lim_{k\to\infty}T_k=\infty$ and
$\left(\wtd \Pi_{T_k}(\cdot)\right)_{k\in\N}$ converges weakly to an invariant probability measure $\pi$ of $\BX$
when $k\to\infty$ .
Then for this sample path, we have
$\int_{\R^n_+}h(\bx)\wtd\Pi_{T_k}(d\bx)\to \int_{\R^n_+}h(\bx)\pi(d\bx)$
for any continuous function $h:\R^n_+\to\R$ satisfying
$|h(\bx)|<K_h(1+\bc^\top\bx)^{\delta}(1+\sum_i(|f_i(\bx)|+|g_i(\bx)|^2))\,,\,\bx\in \R^n_+$,
with $K_h$ a positive constant and $\delta\in[0,\delta_1)$.
\end{lm}
The proofs of Lemmas \ref{lm4.4} and \ref{lm4.7} are given in the Appendix while that of Lemma \ref{lm4.5} is almost the same as that of Lemma \ref{lm2.4} and is left for the reader.
\begin{lm}\label{lm4.6}Let Assumption \ref{a.extn2} be satisfied.
For any initial condition $\BX(0)=\bx\in\R^n_+$,
the family $\left\{\wtd \Pi_t(\cdot), t\geq 1\right\}$ is tight in $\R^n_+$,
and its weak$^*$-limit set, denoted by $\U=\U(\omega)$
is a family of invariant probability measures of $\BX$ with probability 1.
\end{lm}
\begin{proof}
The tightness follows from Lemma \ref{lm4.4}.
The property of the weak$^*$-limit set of normalized occupation measures
was first proved in  \cite[Theorems 4, 5]{SBA11} for compact state spaces and then generalized to a locally compact state space in \cite[Theorem 4.2]{EHS15}. Similar results for general Markov processes can be found in \cite{B14}.
\end{proof}
\begin{lm}\label{lm4.9}
Suppose that Assumption \ref{a.extn2} is satisfied.
Let $\mu\in\M^1$.
For any $\bx\in\R^{n,\circ}_+$,
$$\PP_\bx\Big\{\U(\omega)\subset\Conv(\M_\mu\cup\{\mu\})\Big\}=\PP_\bx\Big\{\U(\omega)=\{\mu\}\Big\}$$
\end{lm}
\begin{rmk}
Note that, since $\Big\{\U(\omega)=\{\mu\}\Big\} \subset \Big\{\U(\omega)\subset\Conv(\M_\mu\cup\{\mu\})\Big\}$, it would be equivalent to prove that $\Big\{\U(\omega)\subset\Conv(\M_\mu\cup\{\mu\})\Big\} = \Big\{\U(\omega)=\{\mu\}\Big\}$  $\PP_\bx$ - a.s. for all $\bx\in\R^{n,\circ}_+$.
\end{rmk}
\begin{proof}
Since $\mu$ satisfies Assumption \ref{a.extn},
it follows from \eqref{ae3.2} that, there are $p^\mu_i>0, i\in I_\mu$ such that
\begin{equation}\label{e2-lm4.9}
\sum_{i\in I_\mu} p^\mu_i\lambda_i(\nu)>0, \nu\in\M_\mu.
\end{equation}
As a result of Lemmas \ref{lm2.3}, \ref{lm4.5} and \ref{lm4.6},
\begin{equation}\label{e10-lm4.9}
\PP_{\bx}\left\{\lim\limits_{t\to\infty}\dfrac1t\int_0^t \left[\dfrac{\sum_ic_iX_i(s)f_i(\BX(s))}{1+\bc^\top\BX(s)}-\dfrac12\dfrac{\sum_{i,j} c_ic_jX_i(s)X_j(s)g_i(\BX(s))g_j(\BX(s))}{(1+\bc^\top\BX(s))^2}\right]ds=0\right\}=1.
\end{equation}
In light of It\^o's formula, it follows from \eqref{e8-lm4.9} and \eqref{e10-lm4.9} that
\begin{equation}\label{e7-lm4.9}
\PP_{\bx}\left\{\limsup_{t\to\infty}\dfrac{\ln X_i(t)}t\leq0,\,\, i=1,\dots,n\right\}\geq\PP_{\bx}\left\{\limsup_{t\to\infty}\dfrac{\ln (1+\bc^\top\BX(t))}t=0\right\}=1,\,\bx\in\R^{n,\circ}_+.
\end{equation}
On the other hand, similarly to \eqref{e8-lm4.9}, we have
\begin{equation}\label{e6-lm4.9}
\PP_{\bx}\left\{\lim_{t\to\infty}\dfrac1t\int_0^tg_i(\BX(s))dE_i(s)=0,\,i=1,\dots,n\right\}=1.
\end{equation}

In view of \eqref{e6-lm4.9} and \eqref{e7-lm4.9}, to prove the lemma, it suffices to show that if the following properties
\begin{itemize}
\item[a)] $\U(\omega)\subset\Conv(\M_\mu\cup\{\mu\})$
\item[b)] \begin{equation}\label{e3-lm4.9}
\lim_{t\to\infty}\dfrac1t\int_0^tg_i(\BX(s))dE_i(s)=0,\,\, i=1,\dots,n
\end{equation}
\item[c)] \begin{equation}\label{e5-lm4.9}
\limsup_{t\to\infty}\dfrac{\ln X_i(t)}t\leq0,\,\, i=1,\dots,n
\end{equation}
\end{itemize}
hold then
$\U(\omega)=\{\mu\}.$

We argue by contradiction. Assume there is a
sequence $\{t_k\}$ with $\lim_{k\to\infty}t_k=\infty$ such that
 $\wtd \Pi_{t_k}(\cdot)$ converges weakly to an invariant probability of the form
$\pi=(1-\rho)\pi_1+\rho\mu$
where $\rho\in(0,1]$ and $\pi_1\in\Conv(\M_\mu)$.
It follows from Lemma \ref{lm4.4} and  \eqref{e2-lm4.9} that
\begin{equation}\label{e4-lm4.9}
\begin{aligned}
\lim_{k\to\infty}\dfrac1{t_k}&\sum_{i\in I_\mu} p^\mu_i\int_0^{t_k}\left(f_i(\BX(s))-\dfrac{\sigma_{ii}g_i^2(\BX(s))}2\right)ds\\
=&\sum_{i\in I_\mu} p^\mu_i\lambda_i(\pi)\\
=&(1-\rho)\sum_{i\in I_\mu}p^\mu_i\lambda_i(\pi_1)+\rho \sum_{i\in I_\mu}p^\mu_i\lambda_i(\mu)\\
=&(1-\rho)\sum_{i\in I_\mu}p^\mu_i\lambda_i(\pi_1) \,\,\text{ (due to Lemma \ref{lm4.1})}\\
>&0.
\end{aligned}
\end{equation}
As a result of \eqref{e3-lm4.9}, \eqref{e4-lm4.9} and It\^o's formula
$$
\begin{aligned}
\lim_{k\to\infty}\sum_{i\in I_\mu}&p^\mu_i\dfrac{\ln X_i(t_k)}{t_k}\\
=&\lim_{k\to\infty}\dfrac1{t_k}\sum_{i\in I_\mu} p^\mu_i\int_0^{t_k}\left[\left(f_i(\BX(s))-\dfrac{\sigma_{ii}g_i^2(\BX(s))}2\right)ds+g_i(\BX(s))dE_i(s)\right]>0
\end{aligned}
$$
which contradicts \eqref{e5-lm4.9}.
This finishes the proof.
\end{proof}
\begin{lm}\label{lm4.8}
Suppose that Assumption \ref{a.extn2} is satisfied.
Let $\mu\in\M^1$.
For any $k\in\N$,  $\eps>0$,
there is $\Delta>0$ such that
$$\PP_\bx\left\{\U(\omega)=\{\mu\} \text{ and }\lim_{t\to\infty}\dfrac{\ln X_i(t)}t=\lambda_i(\mu)<0, i\in I_\mu^c\right\}>1-\eps, \bx\in\K^{k,\Delta}_\mu$$
where
$$\K^{k,\Delta}_\mu:=\{\bx\in\R^{n,\circ}_+, k^{-1}\leq x_i\leq k\text{ for } i\in I_\mu, x_i<\Delta\text{ for }i\in I_\mu^c\}.$$
\end{lm}
\begin{proof}
Let $\tilde U(\bx)$ be the function defined as in the proof of Theorem \ref{thm4.1}.

In view of Lemma \ref{lm4.7}, there is
$\beta>0$ such that
\begin{equation}\label{e2-lm4.8}
\PP_\bx\left\{\max_{i\in I^c_\mu}\{X_i(t)\}>\beta\,,\forall\, t\in[0,n_eT_e]\right\}>\frac{1}{2},\,\bx\in\CH
\end{equation}
where
$$\CH=\{\bx\in\R^n_+: \|\bx\|\leq \hat K_1, \,\max_{i\in I^c_\mu}\{x_i\}\geq1\}.$$
It can be seen that
\begin{equation}\label{e:nu_supp}
\nu(\CH)>0, \nu\in\M\setminus(\M_\mu\cup\{\mu\}).
\end{equation}
By the definition of $\tilde U(\cdot)$, there is
$\Delta>0$ sufficiently small such that
\begin{equation}\label{e3-lm4.8}
\sup_{\bx\in\K^{k,\Delta}_\mu}\{\tilde U(\bx)\}\leq \dfrac{2}{\eps}\inf_{\by\in\R^{n,\circ}_+,x_i\geq\beta, i\in I^c_\mu}\{\tilde U(\by)\}
\end{equation}
In view of \eqref{e3-lm4.8}, since $\{Y(j):=\tilde U_\theta(\BX(jn_eT_e)): j\in\N\}$ is a supermartingale, similar to \eqref{et3.7}, we can obtain
\begin{equation}\label{e4-lm4.8}
\PP_\bx\left\{\max_{i\in I^c_\mu}\{X_i(jn_eT_e)\}< \beta\,\text{ for all } j\in\N\right\}>1-\dfrac\eps2\, \text{ if } \bx\in\K^{k,\Delta}_\mu.
\end{equation}
Let $\tau_\CH=\inf\{t>0: \BX(t)\in\CH\}$.

Now, suppose that
there is  $\bx\in\K^{k,\Delta}_\mu$ such that
\begin{equation}\label{e5-lm4.8}
\PP_\bx\left\{
\limsup_{t\to\infty}\dfrac1t\int_0^t \1_{\{\BX(s)\in \CH\}}ds>0
\right\}>\eps.
\end{equation}
Then
\begin{equation}\label{e6-lm4.8}
\PP_\bx\{\tau_\CH<\infty\}>\eps.
\end{equation}
By the strong Markov property of $\{\BX(t):t\in\R_+\}$,
it follows from \eqref{e2-lm4.8} and \eqref{e6-lm4.8}
that
$$
\PP_\bx\left(\{\tau_\CH<\infty\}\bigcap\left\{\max_{i\in I^c_\mu}\{X_i(t)\}\geq\beta\,\text{ for } t\in[\tau_\CH,\tau_\CH+n_eT_e]\right\}\right)>\frac{1}{2}\eps.
$$
which contradicts \eqref{e4-lm4.8}.
Thus, \eqref{e5-lm4.8} does not hold, that is,
we have
\begin{equation}\label{e7-lm4.8}
\PP_\bx\left\{
\lim_{t\to\infty}\dfrac1t\int_0^t \1_{\{\BX(s)\in \CH\}}ds=0
\right\}>1-\eps, \,\bx\in \K^{k,\Delta}_\mu
\end{equation}
If for an $\omega\in\Omega$,  and a sequence $\{t_j\}$ with $\lim_{j\to\infty}t_j=\infty$,
 $\wtd \Pi_{t_j}(\cdot)$ converges weakly to an invariant probability of the form
$\pi=(1-\rho)\pi_1+\rho\pi_2$
where $\rho\in(0,1]$ and $\pi_1\in\Conv(\M_\mu\cup\{\mu\})$, $\pi_2\in\Conv(\M\setminus(\M_\mu\cup\{\mu\}))$
 then by \eqref{e:nu_supp}
$$\limsup_{j\to\infty}\dfrac1{t_j}\int_0^{t_j} \1_{\{\BX(s)\in \CH\}}ds\geq \pi(\CH)\geq\rho\pi_2(\CH)>0.$$
This inequality, combined with Lemma \ref{lm4.6} and \eqref{e7-lm4.8}, implies that
$$\PP_\bx\left\{
\U(\omega)\subset\Conv(\M_\mu\cup\{\mu\})
\right\}>1-\eps, \,\bx\in \K^{k,\Delta}_\mu.
$$
Lemma \ref{lm4.9} and the above force
\begin{equation}\label{e8-lm4.8}
\PP_\bx\Big\{
\U(\omega)=\{\mu\}
\Big\}>1-\eps, \bx\in\K^{k,\Delta}_\mu.
\end{equation}
In view of Lemma \ref{lm4.4} and \eqref{e8-lm4.8}, we have for $\bx\in\K^{k,\Delta}_\mu$ and for each $i=1,\dots,n$ that
\begin{equation}\label{e9-lm4.8}
\PP_\bx\left\{\lim_{t\to\infty}\dfrac1t\int_0^{t}\left(f_i(\BX(s))-\dfrac{\sigma_{ii}g_i^2(\BX(s))}2\right)ds=\lambda_i(\mu)\right\}>1-\eps.
\end{equation}
The claim of this lemma follows from \eqref{e9-lm4.8}, \eqref{e6-lm4.9} and an application of It\^o's formula.
\end{proof}
\begin{thm}\label{thm4.2}
Suppose that Assumptions \ref{a.nonde}, \ref{a.extn2} and \ref{a.extn3} are satisfied and $\M^1\neq \emptyset$.
Then for any $\bx\in\R^{n,\circ}_+$
\begin{equation}\label{e0-thm4.2}
\sum_{\mu\in\M^1} P_\bx^\mu=1
\end{equation}
where for $\bx\in\R^{n,\circ}_+, \mu\in \M^1$
$$P_\bx^\mu:=\PP_\bx\left\{\U(\omega)=\{\mu\}\,\text{ and }\,\lim_{t\to\infty}\dfrac{\ln X_i(t)}t=\lambda_i(\mu)<0, i\in I_\mu^c\right\}>0.$$
\end{thm}
\begin{proof}
First, suppose that Assumption \ref{a.extn3} is satisfied with nonempty $\M^2$.
Then, there is $\bq=(q_1,\dots,q_n)\in\R^{n,\circ}_+$ such that
$\|\bq\|=1$ and
\begin{equation}\label{e1-thm4.2}
\min_{\nu\in\M^2}\left\{\sum_iq_i\lambda_i(\nu)\right\}>0.
\end{equation}
Using \eqref{e1-thm4.2} and arguing by contradiction, similar to the argument from Lemma \ref{lm4.9},
we can show that with probability 1,
$\U(\omega)$ is a subset of $\Conv(\M)\setminus\Conv(\M^2)$.
In other words,  each invariant probability $\pi\in\U(\omega)$ has the form
$\pi=(1-\rho)\pi_1+\rho\pi_2$
where $\rho\in[0,1), \pi_1\in\Conv(\M^1), \pi_2\in\Conv(\M^2)$.
Let
$k_0>1$ and  for each $\mu\in\M^1$ define
$$\K^0_\mu=\{\bx\in\R^\mu_+: x_i\wedge x_i^{-1}\leq k_0, i\in I_\mu\}.$$

By Lemma \ref{lm4.8}, there are $k>k_0$ and $\Delta>0$ such that
\begin{equation}\label{e3-thm4.2}
\PP_\bx\left\{\lim_{t\to\infty}\dfrac{\ln X_i(t)}t=\lambda_i(\mu)<0, i\in I_\mu^c\right\}>1-\eps
\end{equation}
for all $\mu\in\M^1$ and
$\bx\in\K_\mu^{k,\Delta}.$
Let $\psi(\cdot):\R^n_+\to[0,1]$ be a continuous function satisfying
$$
\psi(\bx)=
\begin{cases}
1 \text{ if }\bx\in\bigcup_{\mu\in\M^1}\K_\mu^0\\
0 \text{ if }\bx\in\R^{n,\circ}_+\setminus \left(\bigcup_{\mu\in\M^1}\K_\mu^{k,\Delta}\right).
\end{cases}
$$
Since $\pi_1(\bigcup_{\mu\in\M^1}\K_\mu^0)>0$ for any $\pi_1\in\Conv(\M^1)$
and $\U(\omega)$ is a subset of $\Conv(\M)\setminus\Conv(\M^2)$ with probability 1,
then we have from Lemma \ref{lm4.6} that
\begin{equation}\label{e4-thm4.2}
\PP_\bx\left\{\liminf_{t\to\infty}\dfrac1t\int_0^t\psi(\BX(s))ds>0\right\}=1,\,\,\bx\in\R^{n,\circ}_+.
\end{equation}
Since $\psi(\bx)=0$ if $\bx\in\R^{n,\circ}_+\setminus \left(\bigcup_{\mu\in\M^1}\K_\mu^{k,\Delta}\right)$,
we deduce from \eqref{e4-thm4.2} that
\begin{equation}\label{e5-thm4.2}
\PP_\bx\left\{\liminf_{t\to\infty}\dfrac1t\int_0^t\1_{\left\{\BX(s)\in \bigcup_{\mu\in\M^1}\K_\mu^{k,\Delta}\right\}}ds>0\right\}=1,\,\,\bx\in\R^{n,\circ}_+.
\end{equation}
Thus, if $\BX(0)\in\R^{n,\circ}_+$ then $\{\BX(s)\}$ will enter $\bigcup_{\mu\in\M^1}\K_\mu^{k,\Delta}$ with probability 1.
This fact, combined with \eqref{e3-thm4.2} and the strong Markov property of $\{\BX(s)\}$, implies
that
$$\sum_{\mu\in\M^1} P_\bx^\mu>1-\eps,\,\bx\in\R^{n,\circ}_+$$
where
$$P_\bx^\mu=\PP_\bx\left\{\U(\omega)=\{\mu\}\text{ and }\lim_{t\to\infty}\dfrac{\ln X_i(t)}t=\lambda_i(\mu)<0, i\in I_\mu^c\right\}.$$
Letting $\eps\to0$ we obtain \eqref{e0-thm4.2}.
The positivity of $P_\bx^\mu$ follows from
\eqref{e3-thm4.2} and the fact that $\{\BX(s)\}$ will visit $\K_\mu^{k,\Delta}$ with a positive probability due to the non degeneracy of the diffusion.	

Next, we consider the case when $\M^1\neq \emptyset$ and $\M^2=\emptyset$.
Then, we claim that $\M^1=\{\bdelta^*\}$
where $\bdelta^*$ be the Dirac measure concentrated on the origin $\0$.
Indeed, if $\M^1$ contains a measure $\mu$ with $\R_+^\mu\ne\{\0\}$,
then $\bdelta^*\in\M_\mu$.
Since $\mu$ satisfies Assumption \ref{a.extn}, in view of \eqref{ae3.2} , $\bdelta^*\in\M^2$ which results in a contradicition.
Thus, $\M=\M^1=\{\bdelta^*\}$.
As a result, $\U(\omega)=\{\bdelta^*\}$ with probability 1.
Then, we can easily deduce with probability 1 that
$$\lim_{t\to\infty}\dfrac{\ln X_i(t)}t=\lambda_i(\bdelta^*)=f_i(\0)-\dfrac{\sigma_{ii}g_i^2(\0)}2<0, i=1,\dots,n$$
since $\bdelta^*$ satisfies \eqref{ae3.1}.
\end{proof}

{\bf Acknowledgments.} The authors thank Michel Bena{\"\i}m for helpful discussions and for sending them his manuscript \cite{B14} which was key in proving the persistence results from this paper. We also thank two anonymous referees for their suggestions and comments which helped improve this manuscript.

\appendix
\section{Proofs for Lemmas in Section 3}
\begin{proof}[Proof of Lemma \ref{lm2.0}]
We restrict our proof for the existence and uniqueness of the solution with initial value $\bx\in \R^{n,\circ}_+$.
If $\bx\in \R^{I,\circ}_+$ for any $I\subset \{1,\cdots,n\}$, the proof carries over.
Let $V(\cdot)$ be defined by \eqref{e:V}. Since $\|p\|\leq\delta_0<1$, it is obvious that
\begin{equation}\label{e1-lm2.0}
\lim\limits_{m\to\infty}\inf\{V(\bx): x_i\vee x_i^{-1}>m\,\text{ for some } i=1,\dots, n\}=\infty.
\end{equation}
Note that
\begin{equation}\label{lv_delta}
\begin{aligned}
\Lom V^{\delta_0}(\bx)=\delta_0V^{\delta_0}(\bx)\bigg[&\dfrac{\sum_ic_ix_if_i(\bx)}{1+\bc^\top\bx}+\dfrac{\delta_0-1}2\dfrac{\sum_{i,j} \sigma_{ij}c_ic_jx_ix_jg_i(\bx)g_j(\bx)}{(1+\bc^\top\bx)^2}\\
&-\sum_i\left(p_if_i(\bx)-\dfrac{p_ig_i^2(\bx)\sigma_{ii}}2\right)+\dfrac{\delta_0}2\sum_{i,j} p_ip_j\sigma_{ij}g_i(\bx)g_j(\bx)\\
&-\delta_0\sum_{i,j}\dfrac{c_ip_ix_i\sigma_{ij}g_i(\bx)g_j(\bx)}{(1+\bc^\top\bx)}\bigg].
\end{aligned}
\end{equation}
Since $\|p\|\leq\delta_0$, we have
\begin{equation}\label{e2.4}
\left|\sum_ip_if_i(\bx)\right|\leq\delta_0\sum_i|f_i(\bx)|
\end{equation}
and
\begin{equation}\label{e2.5}
\begin{aligned}
\dfrac{\delta_0}2&\dfrac{\sum_{i,j} \sigma_{ij}c_ic_jx_ix_jg_i(\bx)g_j(\bx)}{(1+\bc^\top\bx)^2}+\sum_i\dfrac{p_ig_i^2(\bx)\sigma_{ii}}2+\dfrac{\delta_0}2\sum_{i,j} p_ip_j\sigma_{ij}g_i(\bx)g_j(\bx)\\
&-\delta_0\sum_{i,j}\dfrac{c_ip_ix_i\sigma_{ij}g_i(\bx)g_j(\bx)}{(1+\bc^\top\bx)}
\leq 3\delta_0\sum_{i,j} \left(|g_i(\bx)g_j(\bx)\sigma_{ij}|\right).
\end{aligned}
\end{equation}
Applying  \eqref{e2.1}, \eqref{e:H}, \eqref{e2.4} and \eqref{e2.5} to \eqref{lv_delta} one gets 
\begin{equation}\label{e2-lm2.0}
\Lom V^{\delta_0}(\bx)\leq \delta_0HV^{\delta_0}(\bx)\,, \,\bx\in\R^{n,\circ}_+.
\end{equation}
Since the coefficients of \eqref{e:system} are locally Lipschitz,
using \eqref{e1-lm2.0} and \eqref{e2-lm2.0},
it follows from \cite[Theorem 3.5]{RK} that
\eqref{e:system} has a unique solution $\BX(t)$ that remains in $\R^{n,\circ}_+$ almost surely for all $t\geq 0$ whenever $\BX(0)=\bx\in\R^{n,\circ}_+$.
The estimate \eqref{e1-lm2.2} can also be derived from \cite[Theorem 3.5]{RK}.
\end{proof}
\begin{proof}[Proof of Lemma \ref{lm2.2}]
Let $V_\bc(\bx):=(1+\bc^\top\bx)^{\delta_0}$.
By noting that
$$\delta_0 \left|\dfrac{\sum_{i,j} \sigma_{ij}c_ic_jx_ix_jg_i(\bx)g_j(\bx)}{(1+\bc^\top\bx)^2}\right|
\leq \delta_0 \sum_{i,j}|g_i(\bx)g_j(\bx)\sigma_{ij}|
$$
a direct calculation combined with \eqref{e2.1} and \eqref{e:H} shows that
\begin{equation}\label{lvc}
\begin{aligned}
\Lom V_\bc(\bx)=&\delta_0V_\bc(\bx)\bigg[\dfrac{\sum_ic_ix_if_i(\bx)}{1+\bc^\top\bx}+\dfrac{\delta_0-1}2\dfrac{\sum_{i,j} \sigma_{ij}c_ic_jx_ix_jg_i(\bx)g_j(\bx)}{(1+\bc^\top\bx)^2}\bigg]\\
\leq&
\delta_0V_\bc(\bx)\bigg[\dfrac{\sum_ic_ix_if_i(\bx)}{1+\bc^\top\bx}-\dfrac{1}2\dfrac{\sum_{i,j} \sigma_{ij}c_ic_jx_ix_jg_i(\bx)g_j(\bx)}{(1+\bc^\top\bx)^2}+\gamma_b+\delta_0\sum_i(|f_i(\bx)|+g^2_i(\bx))\\
&\qquad\qquad+\delta_0\sum_{i,j} |g_i(\bx)g_j(\bx)\sigma_{ij}|\bigg]-\delta_0V_\bc(\bx)\left[\gamma_b+\delta_0\sum(|f_i(\bx)|+g^2_i(\bx))\right]\\
\leq& \delta_0HV_\bc(\bx)\1_{\{\|\bx\|\leq M\}}-\delta_0V_\bc(\bx)(\gamma_b+\delta_0\sum_i(|f_i(\bx)|+|g_i(\bx)|^2))\\
\leq&\delta_0H\sup_{\|\by\|\leq M}\{V_c(\by)\}-\delta_0V_\bc(\bx)(\gamma_b+\delta_0\sum_i(|f_i(\bx)|+|g_i(\bx)|^2))\,\,\forall \bx\in\R^n_+.
\end{aligned}
\end{equation}

Letting $\eta_k=\inf\{t>0: \|\BX(t)\|\geq k\}$, we have by applying Dynkin's formula to the function $\varphi(\bx,t)=e^{\gamma_b\delta_0 t}V_\bc(\bx)$ and the stopping time $\eta_k\wedge t$ and making use of \eqref{lvc} that
\begin{equation}\label{e2.20}
\begin{aligned}
\E_\bx &e^{\delta_0\gamma_b(\eta_k\wedge t)}V_\bc(\BX(\eta_k\wedge t))\\
&=V_\bc(\bx)+\E_{\bx}\int_0^{\eta_k\wedge t} e^{\delta_0\gamma_b s}\left(\delta_0\gamma_bV_\bc(\BX(s))+\Lom V_\bc (\BX(s))\right)ds\\
&\leq V_\bc(\bx)+\Big(\delta_0H\sup_{\|\by\|\leq M}\{V_c(\by)\}\Big)\E_{\bx}\int_0^{\eta_k\wedge t} e^{\delta_0\gamma_b s}ds\\
&\leq V_\bc(\bx)+H_1e^{\delta_0\gamma_bt} \,\text{ with }
\end{aligned}
\end{equation}
where $H_1:=\gamma_b^{-1}H_1\sup_{\|\by\|\leq M}\{V_c(\by)\}$.
Letting $k\to\infty$ in \eqref{e2.20} together with Fatou's lemma forces
$e^{\delta_0\gamma_bt}\E_\bx V_\bc(\BX(t))\leq V_\bc(\bx)+H_1e^{\delta_0\gamma_bt}$,
which in turn implies
$$\E_\bx V_\bc(\BX(t))\leq H_1+V_\bc(\bx)e^{-\delta_0\gamma_bt}$$
as required.
Another application of Dynkin's formula combined with \eqref{lvc} yields
$$\begin{aligned}
\E_\bx &V_\bc(\BX((\eta_k\wedge t)))\\
&=V_\bc(\bx)+\E_{\bx}\int_0^{\eta_k\wedge t} \Lom V_\bc(\BX(s))ds\\
&\leq V_\bc(\bx)+\delta_0\tilde H_1\E\int_0^{\eta_k\wedge t}ds-\delta_0^2\E_{\bx}\int_0^{\eta_k\wedge t} V_\bc(\BX(s))\left[1+\sum_i(|f_i(\BX(s)|+|g_i(\BX(s)|^2)\right]ds.\end{aligned}
$$
As a result
$$\delta_0^2\E_{\bx}\int_0^{\eta_k\wedge t} V_\bc(\BX(s))\left[1+\sum_i(|f_i(\BX(s)|+|g_i(\BX(s)|^2)\right]ds\leq V_\bc(\bx)+\delta_0\tilde H_1t.$$
If we let $k\to\infty$
we obtain \eqref{e3-lm2.2} with $H_2=\delta_0^{-2}\vee \delta_0^{-1}\tilde H_1$.

Finally, since $\lim_{\|\bx\|\to\infty}V_c(\bx)=\infty$,
it follows easily from \eqref{e2.20} that
$$\lim_{k\to\infty}\PP_\bx\{\eta_k<t\}=0\text{ uniformly on each compact subset of } \R^n_+.$$
The above coupled with the assumption that $f_i(\cdot), g_i(\cdot), i=1,\dots,n$ are locally Lipschitz allow us to modify the proof
of \cite[Theorem 2.9.3]{MAO} by a truncation argument in order to
get the Markov-Feller property of $(\BX(t))$.
\end{proof}
\begin{proof}[Proof of Lemma \ref{lm2.3}]
It suffices to suppose that $\mu$ is ergodic.

Let $\phi(\bx)=[(1+\bc^\top\bx)^{\delta_0}\left(1+\sum_i(|f_i(\bx)|+|g_i(\bx)|^2)\right)$.
Since $\mu$ is invariant, we have
\begin{equation}\label{eq1-lm2.3}
\int_{\R^n_+}(k\wedge\phi(\bx))\mu(d\bx)=\lim\limits_{t\to\infty}\int_{\R^n_+} \E_\bx\big[k\wedge\phi(\BX(t)\big]\mu(d\bx)
\end{equation}
In view of Lemma \ref{lm2.2},
\begin{equation}\label{eq2-lm2.3}
\limsup_{t\to\infty}\int_{\R^n_+} \E_\bx\big[k\wedge\phi(\BX(t)\big]\mu(d\bx)\leq H_2,\,\bx\in\R^n_+
\end{equation}
As a consequence of Fatou's Lemma, it follows from \eqref{eq1-lm2.3} and
\eqref{eq2-lm2.3} that $$\int_{\R^n_+}(k\wedge\phi(\bx))\mu(d\bx)\leq H_2\,\text{ for any }\,k\in\N.$$
Letting $k\to\infty$ and making use of Fatou's Lemma again, we get
$$\int_{\R^n_+}\phi(\bx)\mu(d\bx)\leq H_2.$$
By the strong law of large numbers (see e.g. \cite[Theorem 4.2]{RK})  and the $\mu$-integrability of $\sum_i(|f_i(\bx)|+|g_i^2(\bx)|)$ (due to the inequality above) one gets
\begin{equation}\label{e2.7}
\begin{aligned}
\lim\limits_{t\to\infty}&\dfrac1t\int_0^t \left[\dfrac{c_iX_i(s)f_i(\BX(s))}{1+\sum_ic_iX_i(s)}-\dfrac12\dfrac{\sum_{i,j} c_ic_jX_i(s)X_j(s)g_i(\BX(s))g_j(\BX(s))}{(1+\sum_ic_iX_i(s))^2}\right]ds\\
&=\int_{\R^n_+}\left[\dfrac{c_ix_if_i(\bx)}{1+\bc^\top\bx}-\dfrac12\dfrac{\sum_{i,j} \sigma_{ij}c_ic_jx_ix_jg_i(\bx)g_j(\bx)}{(1+\bc^\top\bx)^2}\right]\mu(d\bx)<\infty\,~~~\PP_{\mu}-\text{a.s.}
\end{aligned}
\end{equation}
and
$$
\begin{aligned}
\lim\limits_{t\to\infty}&\dfrac1t\int_0^t \dfrac{\sum_{i,j} c_ic_jX_i(s)X_j(s)g_i(\BX(s))g_j(\BX(s))}{(1+\sum_ic_iX_i(s))^2}ds\\
&=\int_{\R^n_+}\dfrac{\sum_{i,j} \sigma_{ij}c_ic_jx_ix_jg_i(\bx)g_j(\bx)}{(1+\bc^\top\bx)^2}\mu(d\bx)<\infty\,~~~~\PP_{\mu}-\text{a.s.}
\end{aligned}
$$
The above limit tells us that if we let
\[
Q_t:=\langle L_\cdot,L_\cdot \rangle_t
\]
be the quadratic variation of the local martingale
\[
L_t:= \int_0^t\dfrac{\sum_ic_iX_i(s)g_i(\BX(s))dE_i(s)}{1+\sum_ic_iX_i(s)}
\]
then
\[
\limsup_{t\to\infty}\frac{Q_t}{t} = \int_{\R^n_+}\dfrac{\sum_{i,j} \sigma_{ij}c_ic_jx_ix_jg_i(\bx)g_j(\bx)}{(1+\bc^\top\bx)^2}\mu(d\bx)<\infty\,~~~~\PP_{\mu}-\text{a.s.}
\]
Applying the strong law of large numbers for local martingales  (see \cite[Theorem 1.3.4]{MAO}) one can see that
\begin{equation}\label{e2.8}
\lim\limits_{t\to\infty}\dfrac1t\int_0^t \dfrac{\sum_ic_iX_i(s)g_i(\BX(s))dE_i(s)}{1+\sum_ic_iX_i(s)}=0\,\,\,\,\PP_{\mu}-\text{a.s.}
\end{equation}
In view of \eqref{e2.7}, \eqref{e2.8} and It\^o's formula,
\begin{equation}\label{e2.88}
\lim\limits_{t\to\infty}\dfrac{\ln(1+\bc^\top\BX(t))}t=\int_{\R^n_+}\left[\dfrac{\sum_ic_ix_if_i(\bx)}{1+\bc^\top\bx}-\dfrac12\dfrac{\sum_{i,j} \sigma_{ij}c_ic_jx_ix_jg_i(\bx)g_j(\bx)}{(1+\bc^\top\bx)^2}\right]\mu(d\bx) \,\,\,\,\PP_{\mu} -\text{a.s.}
\end{equation}
A simple contradiction argument coupled with \eqref{e2.88} forces
$$\int_{\R^n_+}\left[\dfrac{\sum_ic_ix_if_i(\bx)}{1+\bc^\top\bx}-\dfrac12\dfrac{\sum_{i,j} \sigma_{ij}c_ic_jx_ix_jg_i(\bx)g_j(\bx)}{(1+\bc^\top\bx)^2}\right]\mu(d\bx)=0.$$
\end{proof}

\begin{proof}[Proof of Lemma \ref{lm2.4}]
Let $V_{M}=\sup\{(1+\bc^\top\bx)^{\delta_0}: \|\bx\|\leq M\}$
and fix $\eps>0.$ Pick
$l_\eps\in \N$ such that
$$\dfrac{\eps (1+\bc^\top\bx)^{\delta_0-\delta}}{K_hH_{2}(1+V_M)}>1\,\text{ for any }
\|\bx\|\geq l_\eps.$$
Let $\phi_l(\cdot):\R^n_+\to[0,1]$ be a continuous function with compact support satisfying  $\phi_l(\bx)=1$ if $\|\bx\|\leq l_\eps$.
One gets the following sequence of inequalities
\begin{equation}\label{e2.10}
\begin{aligned}
\int_{\R^n_+}&\left(1-\phi_l(\bx)\right)|h(\bx)|\Pi^{\bx_k}_{T_k}(d\bx)\\
\leq&K_h\int_{\R^n_+}\left(1-\phi_k(\bx)\right)(1+\bc^\top\bx)^{\delta}(1+\sum_i(|f_i(\bx)|+|g_i(\bx)|^2))\Pi^{\bx_k}_{T_k}(d\bx)\\
\leq&\dfrac{K_h\eps}{K_hH_2(1+V_M)}\int_{\R^n_+}\left(1-\phi_l(\bx)\right)(1+\bc^\top\bx)^{\delta_0}(1+\sum_i(|f_i(\bx)|+|g_i(\bx)|^2))\Pi^{\bx_k}_{T_k}(d\bx)\\
\leq&\eps.
\end{aligned}
\end{equation}
where the last inequality follows by \eqref{e2-lm2.2}.
Similar to \eqref{e2.10} we have from Lemma \ref{lm2.3} that
\begin{equation}\label{e2.11}
\int_{\R^n_+}\left(1-\phi_l(\bx)\right)|h(\bx)|\pi(d\bx)\leq\eps.
\end{equation}
Since $\Pi^{\bx_k}_{T_k}$ converges weakly to $\pi$ we get
\begin{equation}\label{e2.12}
\lim\limits_{k\to\infty}\int_{\R^n_+}\phi_l(\bx)h(\bx)\Pi^{\bx_k}_{T_k}(d\bx)=\int_{\R^n_+}\phi_l(\bx)h(\bx)\pi(d\bx).
\end{equation}
As a consequence of \eqref{e2.10}, \eqref{e2.11} and \eqref{e2.12}
\begin{equation}
\limsup\limits_{k\to\infty}\left|\int_{\R^n_+}h(\bx)\Pi^{\bx_k}_{T_k}(d\bx)-\int_{\R^n_+}h(\bx)\pi(d\bx)\right|\leq2\eps.
\end{equation}
The desired result follows by letting $\eps\to0$.

\end{proof}

\begin{proof}[Proof of Lemma \ref{lm2.5}]
It is easy to show that there exists some $K_2>0$ such that
$$ |y|^k\exp(\theta y)\leq K_2(\exp(\theta_0y)+\exp(-\theta_0y)), k=1,2.$$
for $\theta\in \left[0,\frac{\theta_0}2\right]$, $y\in\R$.
For any $y\in\R$, let $\xi(y)$ be a number lying between $y$ and $0$ such that
$\exp(\xi(y))=\dfrac{e^y-1}y$.
Pick $\theta\in\left[0,\frac{\theta_0}2\right]$
and let $h\in\R$ such that $0\leq \theta+h\leq  \frac{\theta_0}2$. Then
$$\lim\limits_{h\to0}\dfrac{\exp((\theta+h) Y)-\exp(\theta Y)}h= Y\exp(\theta Y)\text{ a.s.}$$
and
$$\left|\dfrac{\exp((\theta+h) Y)-\exp(\theta Y)}h\right|=|Y|\exp(\theta Y+\xi(hY))
\leq 2K_3[ \exp(\theta_0 Y)+\exp(-\theta_0 Y)].$$
By the Lebesgue dominated convergence theorem,
$$\dfrac{d \E \exp(\theta Y)}{d\theta}=\lim\limits_{h\to0}\E\dfrac{\exp((\theta+h) Y)-\exp(\theta Y)}h= \E Y\exp(\theta Y).$$
Similarly,
$$\dfrac{d^2 \E \exp(\theta Y)}{d\theta^2}=\E Y^2\exp(\theta Y).$$
As a result, we obtain
$$\dfrac{d\phi}{d\theta}=\dfrac{\E Y\exp(\theta Y)}{\E \exp(\theta Y)}$$
which implies
	$$\dfrac{d\phi}{d\theta}(0)=\E Y$$
and
$$\dfrac{d^2\phi}{d\theta^2}=\dfrac{\E Y^2\exp(\theta Y)\E \exp(\theta Y)-[\E Y\exp(\theta Y)]^2}{[\E \exp(\theta Y)]^2}.$$
By H{\"o}lder's inequality we have $\E Y^2\exp(\theta Y)\E \exp(\theta Y)\geq[\E Y\exp(\theta Y)]^2$
and therefore $$\dfrac{d^2\phi}{d\theta^2}\geq0\,,\forall\,\theta\in \left[0,\frac{\theta_0}2\right].$$
Moreover,
$$
\begin{aligned}
\dfrac{d^2\phi}{d\theta^2}\leq& \dfrac{\E Y^2\exp(\theta Y)}{\E \exp(\theta Y)}\\
\leq & \dfrac{K_3(\E \exp(\theta_0 Y)+\E \exp(-\theta_0 Y))}{\exp(\theta \E Y)}\\
\leq & \dfrac{K_3(\E \exp(\theta_0 Y)+\E \exp(-\theta_0 Y))}{\exp(-\theta_0|\E Y|)}<K_2
\end{aligned}
$$
for 	some $K_2$ depending only on $K_1$ and $K_3$.
\end{proof}
\begin{proof}[Proof of Lemma \ref{lm2.7}]
Let $K\subset \R_+^{n,\circ}$ be a compact set and let $D$ be an open, relatively compact subset of $\R^{n,\circ}_+$ with smooth boundary such that
 $K\subset D$.
For $\bx\in D$ and $t>0$,
define the measure
$$P_D(t, \bx,\cdot)=\PP_{\bx}\Big(\{\BX(t)\in\cdot\}\cap\{\BX(s)\in D, s\in[0,t]\}\Big).$$
For a bounded continuous function $f:\R^n\mapsto\R$
vanishing outside $D$,
let $u_f(t, x)$ be the solution to
\begin{equation}
\begin{cases}
\dfrac{\partial u}{\partial t}+\Lom u=0 \text{ in } D\times[0,T)\\
u\left(T,x\right)=f(x)\text{ on } D,\\
u(t, x)=0\text{ on } \partial D\times\left[0,T\right].
\end{cases}
\end{equation}
By the Feyman-Kac theorem (see e.g., \cite[Theorem 2.8.2]{MAO}),
$$u_f(t, x)=\int_{D}f(\by)P_D(T-t, \bx,d\by).$$
Under the assumption of nondegeneracy (part (1) of Assumption \ref{a.nonde}),
we deduce from \cite[Theorem 3.16 and its corollary]{AF08} that
$P_D(t,\bx,\cdot)$
has a  density $p_D(t,\bx,\by)$
that is strictly positive and continuous in $(\bx,\by)\in D\times D$.
Since $K$ is compact,
$p_{K, D}(t,\by):=\inf_{\bx\in K}p_D(t,\bx,\by)$
is strictly positive and continuous in $\by\in D.$

For $\by\notin D$, we define $p_{K, D}(t,\by)=0$.
Let $m_{K,D}(\cdot)$ be the measure
whose density is $p_{K, D}(T,\by)$.
For any $\bx\in K$  and a measurable $B\subset\R^{n,\circ}_+$, we have
$$P(t,\bx, B)\geq P_D(T,\bx,B)\geq m_{K,D}(B).$$
Thus, $K$ is a petite set for the Markov chain $\{\BX(kT), k\in\N\}$.
On the other hand,
since  $p_D(t,\bx,\by)$ is strictly positive for any $D$,
we note that
\begin{equation}\label{positive-den}
P(T, \bx, B)>0\,\text{ for any set }\, B\,\text{ whose Lebesgue measure is nonzero}.
\end{equation}
Thus, $\{\BX(kT), k\in\N\}$ is irreducible.
Morever,  it is easy to derive from \eqref{positive-den} that
there are no disjoint subsets of $\R^{n}_+\setminus\{\mathbf{0}\}$, denoted by $A_0, \dots, A_{d-1}$ with some $d>1$
such that for any $\bx\in A_i$,
$$P(T, \bx, A_j)=1 \text{ where } j=i+1 \text{ (mod } d).$$
As a result, the Markov chain $\{\BX(kT), k\in\N\}$ is aperiodic.
\end{proof}

\section{Proofs for Lemmas in Section 4}
\begin{proof}[Proof of Lemma \ref{lm4.4}]
Computations similar to those used to prove \eqref{lvc}, give us
\begin{equation}\label{e1-lm4.4}
\begin{aligned}
\Lom (1+\bc^\top\bx)^{\delta_1}\leq&\delta_1\tilde H_2-\delta_1(1+\bc^\top\bx)^{\delta_1}(\gamma_b+\delta_1\sum_i(|f_i(\bx)|+|g_i(\bx)|^2)),\, \bx\in\R^n_+
\end{aligned}
\end{equation}
for some $\tilde H_2>0$.
Equation \eqref{e1-lm4.4} together with It\^o's formula implies
\begin{equation*}
\begin{aligned}
\dfrac{(1+\bc^\top\BX(t))^{\delta_1}}t=&\dfrac{(1+\bc^\top\BX(0))^{\delta_1}}t
+\dfrac1t\int_0^t \Lom (1+\bc^\top\BX(s))^{\delta_1}ds\\
&+\dfrac1t\int_0^t \dfrac{\sum_ic_iX_i(s)g_i(\BX(s))dE_i(s)}{(1+\bc^\top\BX(s))^{1-\delta_1}}ds\\
\leq&\dfrac{(1+\bc^\top\BX(0))^{\delta_1}}t+\delta_1\tilde H_2+\dfrac1t\int_0^t \dfrac{\sum_ic_iX_i(s)g_i(\BX(s))dE_i(s)}{(1+\bc^\top\BX(s))^{1-\delta_1}}\\
&-\delta_1\dfrac1t\int_0^t(1+\bc^\top\BX(s))^{\delta_1}(\gamma_b+\delta_1\sum_i(|f_i(\BX(s))|+|g_i(\BX(s))|^2))ds.
\end{aligned}
\end{equation*}
Since $\dfrac{(1+\bc^\top\BX(t))^{\delta_1}}t\geq 0$ the above yields
\begin{equation}\label{e2-lm4.4}
\begin{aligned}
\dfrac{\delta_1}{2t}\int_0^t&(1+\bc^\top\BX(s))^{\delta_1}(\gamma_b+\delta_1\sum_i(|f_i(\BX(s))|+|g_i(\BX(s))|^2))\\
\leq&\dfrac{(1+\bc^\top\BX(0))^{\delta_1}}t+\delta_1\tilde H_2+\dfrac1t\int_0^t \dfrac{\sum_ic_iX_i(s)g_i(\BX(s))dE_i(s)}{(1+\bc^\top\BX(s))^{1-\delta_1}}\\
&-\dfrac{\delta_1}{2t}\int_0^t(1+\bc^\top\BX(s))^{\delta_1}(\gamma_b+\delta_1\sum_i(|f_i(\BX(s))|+|g_i(\BX(s))|^2))ds.
\end{aligned}
\end{equation}
For each $i=1,\dots,n$, the quadratic variation of
$$\int_0^t \frac{c_iX_i(s)g_i(\BX(s))dE_i(s)}{(1+\bc^\top\BX(s))^{1-\delta_1}}$$
is
$$Q_t:=\int_0^t \frac{[c_iX_i(s)g_i(\BX(s))]^2\sigma_{ii}}{(1+\bc^\top\BX(s))^{2-2\delta_1}}ds.$$
We have the following estimate for each $i=1,\dots,n$
\begin{align*}
\frac{[c_ix_i(s)g_i(\bx)]^2\sigma_{ii}}{(1+\bc^\top\bx)^{2-2\delta_1}}\leq&
(1+\bc^\top\bx)^{2\delta_1}g_i^2(\bx)\sigma_{ii} \\
\leq&K_i\left[(1+\bc^\top\bx)^{\delta_1}(\gamma_b+\delta_1\sum_i(|f_i(\bx)|+|g_i(\bx)|^2))\right] \\
\end{align*}
where due to Assumption \ref{a.extn2}
$$K_i=\sup_{\bx\in\R^n_+}\left\{\dfrac{(1+\bc^\top\bx)^{\delta_1}g_i^2(\bx)\sigma_{ii}}{\gamma_b+\delta_1\sum_i(|f_i(\bx)|+|g_i(\bx)|^2)}\right\}<\infty.$$
Thus,
\begin{equation}\label{e5-lm4.4}
\int_0^\infty \frac{dQ_t}{(1+A_t)^2}\leq \int_0^\infty K_i\frac{dA_t}{(1+A_t)^2}= K_i <\infty \text{ a.s.}
\end{equation}
where $$A_t:=  \int_0^u(1+\bc^\top\BX(s))^{\delta_1}(\gamma_b+\delta_1\sum_i(|f_i(\BX(s))|+|g_i(\BX(s))|^2))ds.$$

On the other hand
\begin{equation}\label{e6-lm4.4}
\lim_{t\to\infty} A_t\geq \lim_{t\to\infty}\gamma_bt=\infty\text{ a.s.}
\end{equation}

By \eqref{e5-lm4.4}, \eqref{e6-lm4.4} we can use the strong law of large numbers for local martingales (see \cite[Theorem 1.3.4]{MAO} in order to obtain for each $i=1,\dots,n$ that
\begin{equation}\label{e3-lm4.4}
\lim_{t\to\infty}\dfrac{\int_0^t c_iX_i(s)g_i(\BX(s)(1+\bc^\top\BX(s))^{\delta_1-1}dE_i(s)}
{\int_0^t(1+\bc^\top\BX(s))^{\delta_1}(\gamma_b+\delta_1\sum_i(|f_i(\BX(s))|+|g_i(\BX(s))|^2))ds}=0\,\,\,\PP_\bx-\text{ a.s.}.
\end{equation}
This implies
\begin{equation}\label{e4-lm4.4}
\begin{aligned}
\limsup_{t\to\infty}\bigg[&\dfrac1t\int_0^t \dfrac{\sum_ic_iX_i(s)g_i(\BX(s))dE_i(s)}{(1+\bc^\top\BX(s))^{1-\delta_1}}\\
&-\dfrac{\delta_1}{2t}\int_0^t(1+\bc^\top\BX(s))^{\delta_1}(\gamma_b+\delta_1\sum_i(|f_i(\BX(s))|+|g_i(\BX(s))|^2))ds\bigg]\leq0\,\,\,\PP_\bx-\text{ a.s.}
\end{aligned}
\end{equation}
Applying \eqref{e4-lm4.4} to \eqref{e2-lm4.4} we get
\begin{equation}\label{e8-lm4.4}
\limsup_{t\to\infty}\dfrac{\delta_1}{2t}\int_0^t(1+\bc^\top\BX(s))^{\delta_1}(\gamma_b+\delta_1\sum_i(|f_i(\BX(s))|+|g_i(\BX(s))|^2))ds
\leq\delta_1\tilde H_2 \,\,\,\PP_\bx-\text{ a.s.}
\end{equation}
Similar to the proof of \eqref{e2.8}, we can obtain \eqref{e8-lm4.9} from \eqref{e8-lm4.4} and the strong law of large numbers for local martingales. The proof is complete.
\end{proof}
\begin{proof}[Proof of Lemma \ref{lm4.7}]
Let $\hat K_1$ be sufficiently large such that
$(1+\bc^\top\bx)^{\delta_1}>2\hat K$ if $\|\bx\|\geq\hat K_1.$
By Lemma \ref{lm4.4},
$$\limsup_{t\to\infty}\dfrac1t\int_0^t\1_{\{\|\BX(s)\|>\hat K_1\}}ds\leq \dfrac1{2\hat K}\limsup_{t\to\infty}\dfrac1t\int_0^t(1+\bc^\top\BX(s))^{\delta_1}ds\leq \dfrac1{2\hat K}\hat K=\dfrac12\,\,\,\,\PP_\bx\text{-a.s.}, \bx\in\R^n_+.$$
which implies \eqref{e1-lm4.7}.

Next, we prove \eqref{e2-lm4.7}.
Fix $i\in\{1,\dots,n\}$ and define $\tilde V_i=\dfrac{1+\bc^\top\bx}{x_i^{\delta_0}}$
on $\{\bx\in\R^n_+: x_i>0\}$.
Similar to \eqref{e2-lm2.0}, it can be shown that
$$\Lom \tilde V_i^{\delta_0}(\bx)\leq \delta_0H\tilde V_i^{\delta_0}(\bx)\,, \bx\in\R^n_+, x_i>0.$$
Let $\zeta_k:=\inf\{t>0: X_i(t)^{-1}\vee\|\BX(t)\|>k\}$.
We have by Dynkin's formula that
$$
\begin{aligned}
\E_\bx  \tilde V_i(\BX((n_eT_e)\wedge\eta_k))\leq&
\tilde V_i(\bx)+\delta_0 H\E_\bx\int_0^{(n_eT_e)\wedge\eta_k} \Lom \tilde V_i(\BX(s))ds\\
 \leq&\tilde V_i(\bx)+\delta_0 H\int_0^{n_eT_e} \E_\bx \tilde V_i(\BX(s\wedge\eta_k))ds.
\end{aligned}
$$
Using Gronwall's inequality yields 
\begin{equation}\label{e3-lm4.7}
\E_\bx  \tilde V_i(\BX((n_eT_e)\wedge\eta_k))\leq  \tilde V_i(\bx)\exp(\delta_0 Hn_eT_e)\,, \bx\in\R^n_+, x_i>0.
\end{equation}
Let $k_1\in\N$ sufficiently large such that 
\begin{equation}\label{e4-lm4.7}
\tilde V_i(\by)> \dfrac1{\eps_1}\sup_{\|\bx\|\leq\hat K_1, x_i\geq\eps_2}\left\{\tilde V_i(\bx)\right\}\exp(\delta_0 Hn_eT_e) \text{ for all } \by\in\R^n_+, y_i^{-1}\vee\|\by\|>k_1.
\end{equation}
It follows from \eqref{e3-lm4.7} and \eqref{e4-lm4.7} that 
$$\PP_\bx\{\eta_{k_1}<n_eT_e\}\leq \dfrac{\E_\bx  \tilde V_i(\BX((n_eT_e)\wedge\eta_{k_1}))}{\inf\{\tilde V_i(\by):\by\in\R^n_+, y_i^{-1}\vee\|\by\|>k_1\}}\leq\eps_1,\text{ for }\,\bx\in\R^n_+,\|\bx\|\leq\hat K_1, x_i\geq\eps_2.$$
Now inequality \eqref{e2-lm4.7} follows by straightforward computations.
\end{proof}
\bibliographystyle{amsalpha}
\bibliography{Kolmogorov}
\end{document}